\documentclass[reqno, a4paper]{amsart}
\usepackage[british]{babel}
\usepackage{graphicx}
\usepackage{amsfonts}
\usepackage{amssymb}
\usepackage{amsthm}
\usepackage{amsmath}
\usepackage[all]{xy}
\usepackage{mathabx}
\usepackage{hyperref}
\usepackage[T1]{fontenc} 

\theoremstyle{plain}
\newtheorem{theorem}[subsection]{Theorem}
\newtheorem{lemma}[subsection]{Lemma}
\newtheorem{proposition}[subsection]{Proposition}
\newtheorem{corollary}[subsection]{Corollary}

\theoremstyle{definition}
\newtheorem{definition}[subsection]{Definition}
\newtheorem{remark}[subsection]{Remark}
\newtheorem{notation}[subsection]{Notation}
\newtheorem{example}[subsection]{Example}
\newtheorem{convention}[subsection]{Convention}

\newcommand{\defn}{\textbf}

\newcommand{\comp}{\raisebox{0.2mm}{\ensuremath{\scriptstyle{\circ}}}}
\newcommand{\del}{\partial}
\renewcommand{\implies}{$\Rightarrow$}

\newcommand{\iso}{\cong}
\newcommand{\noproof}{\hfill \qed}
\newcommand{\plus}{$^{\textrm{+}}$}

\renewcommand{\H}{\ensuremath{\mathrm{H}}}

\renewcommand{\ker}{\ensuremath{\mathrm{Ker\,}}}
\newcommand{\K}{\ensuremath{\mathrm{K}}}

\newcommand{\op}{\ensuremath{\mathrm{op}}}
\newcommand{\R}{\ensuremath{\mathrm{R}}}

\newcommand{\A}{\ensuremath{\mathcal{A}}}
\newcommand{\B}{\ensuremath{\mathcal{B}}}

\newcommand{\D}{\ensuremath{\mathcal{D}}}
\newcommand{\E}{\ensuremath{\mathcal{E}}}
\newcommand{\F}{\ensuremath{\mathcal{F}}}
\renewcommand{\P}{\ensuremath{\mathcal{P}}}

\newcommand{\Sq}{\ensuremath{\mathsf{S}^{+}_{q}\!}}
\newcommand{\Ss}{\ensuremath{\mathsf{S}^{+}_{s}\!}}

\renewcommand{\AA}{\ensuremath{\mathbb{A}}}
\newcommand{\BB}{\ensuremath{\mathbb{B}}}
\newcommand{\GG}{\ensuremath{\mathbb{G}}}
\newcommand{\ZZ}{\ensuremath{\mathbb{Z}}}

\newcommand{\Ab}{\ensuremath{\mathsf{Ab}}}
\newcommand{\Arr}{\ensuremath{\mathsf{Arr}}}
\newcommand{\Arrn}{\ensuremath{\mathsf{Arr}^{n}\!}}
\newcommand{\Arrnn}{\ensuremath{\mathsf{Arr}^{n+1}\!}}

\newcommand{\Cub}{\ensuremath{\mathsf{Cub}}}
\newcommand{\Cubn}{\ensuremath{\mathsf{Cub}^n\!}}
\newcommand{\Cubnn}{\ensuremath{\mathsf{Cub}^{n+1}\!}}

\newcommand{\Ext}{\ensuremath{\mathsf{Ext}}}
\newcommand{\Extn}{\ensuremath{\mathsf{Ext}^{n}\!}}
\newcommand{\Fun}{\ensuremath{\mathsf{Fun}}}
\newcommand{\Gp}{\ensuremath{\mathsf{Gp}}}

\newcommand{\Loop}{\ensuremath{\mathsf{Loop}}}

\newcommand{\Top}{\ensuremath{\mathsf{Top}}}

\newcommand{\arr}{\ensuremath{\mathsf{arr}}}

\newcommand{\gp}{\ensuremath{\mathsf{gp}}}

\newdir{>}{{}*:(1,-.2)@^{>}*:(1,+.2)@_{>}}
\newdir{<}{{}*:(1,+.2)@^{<}*:(1,-.2)@_{<}}

\newbox\pullbackbox
\setbox\pullbackbox=\hbox{\xy 0;<1mm,0mm>: \POS(4,0)\ar@{-} (0,0) \ar@{-} (4,4)
\endxy}
\def\pullback{\copy\pullbackbox}

\newbox\pushoutbox
\setbox\pushoutbox=\hbox{\xy 0;<1mm,0mm>: \POS(0,4)\ar@{-} (0,0) \ar@{-} (4,4)
\endxy}

\begin{document}

\author{Tomas Everaert}
\author{Julia Goedecke}
\author{Tim Van~der Linden}

\address{Vakgroep Wiskunde, Vrije Universiteit Brussel, Pleinlaan 2, 1050 Brussel, Belgium\newline \indent Institut de recherche en math\'ematique et physique, Universit\'e catholique de Louvain, chemin du cyclotron~2 bte L7.01.02, 1348 Louvain-la-Neuve, Belgium\newline \indent Queens' College, University of Cambridge, United Kingdom\newline \indent 
Centro de Matem\'atica da Universidade de Coimbra, 3001--454 Coimbra, Portugal}

\email{teveraer@vub.ac.be}
\email{julia.goedecke@cantab.net}
\email{tim.vanderlinden@uclouvain.be}

\thanks{The first author's research was supported by Fonds voor Wetenschappelijk Onderzoek (FWO-Vlaanderen). The second author's research was supported by the FNRS grant \emph{Cr\'edit aux chercheurs} 1.5.016.10F. The third author works as \emph{charg\'e de recherches} for Fonds de la Recherche Scientifique--FNRS. His research was supported by Funda\c c\~ao para a Ci\^encia e a Tecnologia (grant number SFRH/BPD/38797/2007) and by CMUC at the University of Coimbra}

\title[Resolutions, higher extensions and the relative Mal'tsev axiom]{Resolutions, higher extensions and\\ the relative Mal'tsev axiom}

\begin{abstract}
We study how the concept of higher-dimensional extension which comes from categorical Galois theory relates to simplicial resolutions. For instance, an augmented simplicial object is a resolution if and only if its truncation in every dimension gives a higher extension, in which sense \emph{resolutions are infinite-dimensional extensions} or \emph{higher extensions are finite-dimensional resolutions.} We also relate certain stability conditions of extensions to the Kan property for simplicial objects. This gives a new proof of the fact that a regular category is Mal'tsev if and only if every simplicial object is Kan, using a relative setting of extensions. 
\end{abstract}



\dedicatory{Dedicated to James Gray and Tamar Janelidze on the occasion of their wedding}

\maketitle

\section*{Introduction}\label{Section-Introduction}
The concept of \emph{higher-dimensional extension} first appeared in the approach to non-abelian homological algebra based on categorical Galois theory in semi-abelian categories. In that context \emph{centrality} of higher extensions plays a very important role, but we do not treat this aspect in the current paper. We rather focus on stability conditions of the higher extensions themselves.

A major point of this article is that certain properties of simplicial objects and simplicial resolutions are actually properties of the induced cubes and higher extensions. As a consequence, some proofs (see, for instance, Proposition~\ref{Proposition-Relative-Maltsev-Converse}) which are rather technical when considered from the simplicial point of view become almost trivial when higher extensions are used instead.

\subsection*{Background on higher (central) extensions}
Already the article~\cite{Brown-Ellis} written by R.~Brown and G.~J.~Ellis on higher Hopf formulae for groups was based on a notion of higher extension of groups. Following G.~Janelidze's ideas set out in~\cite{Janelidze:Double, Janelidze:Hopf-talk} and extending the theory from~\cite{Janelidze-Kelly}, \emph{higher-di\-men\-sion\-al central extensions} were introduced alongside a general categorical concept of \emph{higher-dimensional extension} in~\cite{EGVdL} to study homology in semi-abelian categories. The theory presented there allows for an interpretation of the canonical comonadic homology objects induced by the reflection of a semi-abelian variety to a subvariety in terms of (higher) Hopf formulae, generalising those obtained in~\cite{Brown-Ellis} to contexts beyond the case of abelianisation of groups. For instance, if~$B$ is a loop and $B\cong A/K$ a projective presentation of~$B$, then
\[
\H_{2}(B,\gp)\cong\frac{K\cap [A,A,A]}{[K,A,A]},
\]
where $\H_{2}(B,\gp)$ is the second homology of $B$ relative to the category of groups (i.e., with coefficients in the reflector $\gp\colon{\Loop\to\Gp}$) and the brackets on the right hand side are associators~\cite{EverVdL4}.

The article~\cite{EGVdL} gives calculations of the homology objects for groups vs.\ abelian groups, rings vs.\ zero rings, precrossed modules vs.\ crossed modules, Lie algebras vs.\ modules, groups vs.\ groups of a certain nilpotency or solvability class, etc., in all dimensions. This approach to homology was extended to cover other examples~\cite{Everaert-Gran-TT,Everaert-Gran-nGroupoids} and several theoretical perspectives were explored: slightly different approaches~\cite{EverHopf, Tomasthesis, Janelidze:Hopf}, links with relative commutator theory~\cite{EveraertCommutator, EGAM, EverVdL4, EverVdLRCT}, first steps towards an interpretation of cohomology~\cite{Gran-VdL, RVdL}, the characterisation of higher central extensions~\cite{EverVdL3}, and satellites~\cite{Juliathesis, GVdL2}.

This gives an indication of the importance of higher central extensions in non-abelian homological algebra, in particular in homology and cohomology of non-abelian algebraic objects. However, they could not exist without higher extensions themselves, and in this paper we examine certain stability conditions that higher extensions may have. This leads to strong results on simplicial objects, which of course also play an important role in the study of homology.

\subsection*{Higher extensions} 
Classically, one-dimensional extensions are just regular epimorphisms in a regular category $\A$, which, in the varietal case, are exactly the surjections. Denoting by $\E$ the class of extensions in $\A$, a \defn{double extension} is a commutative square 
\[ 
\vcenter{\xymatrix{A_1 \ar[r]^-{f_1} \ar[d]_-{a} & B_1 \ar[d]^-{b}\\ 
A_0 \ar[r]_-{f_0} & B_0}} 
\] 
in $\A$ where the morphisms $a$, $b$, $f_1$, $f_0$ and the universally induced  morphism $\langle a,f_1\rangle\colon{A_1\to A_0\times_{B_0} B_1}$ to the pullback of $b$ and $f_0$ are in $\E$. We denote the class of double extensions thus obtained by $\E^{1}$. Of course this definition does not depend on the exact nature of 
one-dimensional extensions, so it can be used for any (reasonable) class 
of morphisms $\E$. In particular, it can be iterated to give $n$-fold 
extensions for any $n\geq 2$: then all the arrows in the induced diagram are $(n-1)$-fold extensions. We write $\Ext\A$ for the full subcategory of the 
category of arrows $\Arr\A$ in $\A$ determined by the extensions, and 
similarly $\Extn\A$ for the full subcategory of the category of $n$-fold 
arrows $\Arrn\A=\Arr\Arr^{n-1}\!\A$ determined by the $n$-fold extensions. We denote the class of $(n+1)$-fold extensions by $\E^{n}$. By treating extensions axiomatically, as described below, we can deal with the pairs~$(\Extn\A,\E^n)$ just like the ``base case'' $(\A,\E)$, since such a pair is just another example of a category with a class of extensions. This makes the statements and proofs of many results much easier, and also clarifies in which other situations the results may hold.

\subsection*{Axioms for extensions}
Treating extensions axiomatically (rather than ad hoc, as in~\cite{EGVdL}) has the following advantage. Because the set of axioms is such that it ``goes up'' to higher dimensions, as first formulated by T.~Everaert in \cite{Tomasthesis} and~\cite{EverHopf}, it allows a simultaneous treatment of extensions in all dimensions without having to remember which dimension is currently needed. 

The main list of axioms for a class of extensions~$\E$ in a category~$\A$ considered in this paper is the following.
\begin{enumerate}
\item[(E1)] $\E$ contains all isomorphisms;
\item[(E2)] pullbacks of morphisms in $\E$ exist in $\A$ and are in $\E$;
\item[(E3)] $\E$ is closed under composition;
\item[(E4)] if  $g\comp f\in \E$ then $g\in\E$ (right cancellation);
\item[(E5)] the \defn{$\E$-Mal'tsev axiom}: any split epimorphism of extensions
\[
\xymatrix{A_1 \ar@<.5ex>[r]^-{f_1} \ar[d]_-{a} & B_1 \ar[d]^-{b} \ar@<.5ex>[l]\\
A_0 \ar@<.5ex>[r]^-{f_0} & B_0 \ar@<.5ex>[l]}
\]
in $\A$ is a double extension.
\end{enumerate}
These axioms come in slightly different flavours and are not all treated at once. The first three, (E1)--(E3), go up to higher dimensions without help of the others and already imply the important fact that \emph{higher extensions are symmetric}. Axioms (E1)--(E5) are the setting of Section~\ref{Section-Relative-Maltsev}. In fact, (E5) is equivalent to~(E4) applied to $(\Ext\A,\E^1)$ and implies axiom (E4) for $(\A,\E)$. The axiom (E5) in its absolute form comes from D.~Bourn's~\cite{Bourn2003}; see also~\cite{Bourn1996}.

\subsection*{Resolutions vs.\ extensions}
In Section~\ref{Section-Resolutions-and-Extensions} we assume that the pair $(\A,\E)$ satisfies axioms (E1)--(E3). We compare higher extensions satisfying these axioms to \emph{simplicial $\E$-resolutions}, which are augmented simplicial objects in which all comparison morphisms to the simplicial kernels are morphisms in~$\E$. Truncating an augmented simplicial object induces higher dimensional arrows, and we prove in Theorem~\ref{Theorem-Resolution-is-Cube} that the augmented simplicial object is an $\E$-resolution if and only if each of these truncations gives rise to a higher dimensional extension. In this sense
\begin{center}
\textit{resolutions are infinite-dimensional extensions}\\
or\\
\textit{higher extensions are finite-dimensional resolutions}.
\end{center}
This is, in fact, also how they are used in practice, for example in~\cite{Brown-Ellis} or~\cite{EGVdL}.

\subsection*{The Kan property and Mal'tsev conditions}
In Section~\ref{Section-Relative-Maltsev} we work with a pair~$(\A,\E)$ satisfying (E1)--(E5). In fact, under (E1)--(E4) we prove that 
\begin{center}
\textit{(E5) holds $\Leftrightarrow$ every simplicial object in $\A$ is $\E$-Kan}
\end{center}
(Theorem~\ref{Theorem-Relative-Maltsev}). This justifies calling (E5) the relative Mal'tsev axiom, as it is well known that a regular category $\A$ is Mal'tsev if and only if every simplicial object in $\A$ is Kan \cite[Theorem~4.2]{Carboni-Kelly-Pedicchio}. As a first indication on the usefulness of the relative Kan property we prove that
\begin{center}
\textit{contractible + $\E$-Kan \implies\ $\E$-resolution} 
\end{center}
for augmented $\E$-simplicial objects (Proposition~\ref{Proposition-Contractible-plus-Kan-is-Resolution}). Here an $\E$-simplicial object is one in which all faces $\del_i$ are extensions, and such an object is $\E$-Kan when all comparison morphisms to the universal horn objects are in $\E$.

\subsection*{Weaker conditions on extensions}
Axioms (E1) and (E4) together imply that all split epimorphisms are extensions. However, this is not the case in all examples of interest, as for instance T.~Janelidze's \emph{relative homological} and \emph{relative semi-abelian categories}~\cite{Tamar_Janelidze,Tamar_Janelidze_Semiabelian}. The connection with her work will be developed in a forthcoming article.

\section{Axioms for extensions}\label{Section-Axioms-for-Extensions}
We treat the concept of \emph{higher-dimensional extension}~\cite{Tomasthesis, EverHopf, EGVdL} in an axiomatic manner, recalling the basic definitions and proving fundamental properties: symmetry, and the axioms of extensions going up to higher dimensions (Proposition~\ref{Proposition-Extensions}).

\subsection*{Higher-dimensional arrows}
To understand higher extensions, we must first define what we mean by a \emph{higher-dimensional arrow}. As these play a very important role throughout the paper, we shall take some time to really understand these objects.

To set up a convenient numbering system for our higher-dimensional arrows, we consider the natural numbers by their standard (von Neumann) construction and write $0=\emptyset$ and $n=\{0,\dots, n-1\}$ for $n\geq 1$. We write $2^{n}$ for the power-set of~$n$. Recall that $2^{n}$ is a category with an arrow $S\to T$ for each inclusion $S\subseteq T$ of subsets $S$, $T\subseteq n$. Clearly $2^{1}=2$, the category generated by a single morphism ${0\to 1}$, is an obvious ``template'' for an arrow in a category.

\begin{definition}
The category $\Arrn\A$ consists of \defn{$n$-dimensional arrows} in~$\A$: $\Arr^{0}\!\A=\A$, $\Arr^{1}\!\A=\Arr\A$ is the category of arrows $\Fun(2^{\op},\A)=\A^{2^{\op}}$, and $\Arrnn\A=\Arr\Arrn\A$.
\end{definition}

\begin{example}
A zero-fold arrow is just an object of $\A$, a one-fold arrow is given by an arrow in $\A$, while a two-fold arrow $A$ is a commutative square in~$\A$ with a specified direction:
\begin{equation}\label{Numbering}
\vcenter{\xymatrix{A_{2} \ar[r]^{a_1} \ar[d]_{a_0} \ar@{}[dr]|{\Rightarrow} & A_{1} \ar[d]^{a_0^{1}} \\
A_{\{1\}} \ar[r]_{a_0^{\{1\}}} & A_{0}.}}
\end{equation}
This particular numbering of the objects and arrows will become clear below, after Definition~\ref{Definition-n-cube}. Similarly an $n$-fold arrow is a commutative $n$-cube in~$\A$ with specified directions. By definition a morphism (a natural transformation) between $n$-fold arrows is also an $(n+1)$-fold arrow.
\end{example}

Notice that, by induction, we have an isomorphism
\[
\Arrn\A\iso \A^{2^{\op}\times\cdots\times2^{\op}}.
\]
However, in the step which says that ``a functor ${2^{\op}\to\A^{2^{\op}}}$ corresponds to a functor ${2^{\op}\times 2^{\op}\to \A}$'' (and the higher versions of this) we may easily lose sight of the direction of the arrow, as $2^{\op}\times 2^{\op}$ is of course symmetric. This leads to the concept of the \emph{$n$-cube} corresponding to an $n$-fold arrow, which we shall make more precise and connect to the issue above. Later we shall see that distinguishing between a cube and an arrow with directions is often not as important for our purposes as it may first seem.

\begin{definition}\label{Definition-n-cube}
Let $n\geq 0$. We define an \defn{$n$-cube} in $\A$ to be a functor 
\[
A\colon (2^{n})^{\op}\to \A.
\]
A morphism between two $n$-cubes $A$ and $B$ in $\A$ is a natural transformation $f\colon {A\to B}$. We write $\Cubn\A$ for the corresponding category.
\end{definition}

Thus an $n$-cube is a diagram of a specified shape in $\A$. Clearly a zero-cube is just an object of $\A$ and a one-cube is a morphism in $\A$, so we have $\Cub^0\!\A=\Arr^0\!\A=\A$ and $\Cub^1\!\A=\Arr\A$. A two-cube is a commutative diagram as above, but (a priori) without a specified direction. 

Notice that $2\times2\iso 2^{2}$ and similarly $2\times2^{n}\iso2^{n+1}$, but these isomorphisms are not unique. Roughly speaking, the extra $1$ can be inserted either ``at the bottom'' or ``at the top'' or even ``somewhere in the middle'', and this determines how the new object is numbered. From the existence of these isomorphisms we see that we can view every $n$-fold arrow as an $n$-cube, by replacing the directions with a specific numbering, and that the two categories $\Arrn\A$ and $\Cubn\A$ are isomorphic---but there are several possible isomorphisms which reflect the different ways a direction may correspond to the numbering of the objects. Also, a morphism between $n$-cubes can be viewed as an $(n+1)$-cube. Conversely an $(n+1)$-cube can be considered as an arrow between $n$-cubes in $n+1$ different ways. 

Having chosen one of the isomorphisms ${\Arrn\A\to \Cubn\A}$ mentioned above, we may number an $n$-fold arrow by viewing it as an $n$-cube. If $A$ is an $n$-fold arrow and $S$ and $T$ are subsets of $n$ such that $S\subseteq T$, we write $A_S$ for the image $A(S)$ of $S$ by the functor $A$ and $a_S^T\colon A_T\to A_S$ for the image $A(S\subseteq T)$ of $S\subseteq T$. If~$f\colon {A\to B}$ is a morphism between $n$-fold arrows, we write $f_S\colon A_S\to B_S$ for the $S$-component of the natural transformation~$f$. Moreover, in order to simplify our notations, we write $a_i$ instead of $a^n_{n\backslash \{i\}}$, for $0\leq i \leq n-1$. (See the picture of a double extension~\eqref{Numbering} above for an example.)

\begin{convention}\label{Convention-Iso-Cubes-Arrows}
As mentioned above, there are several different isomorphisms between $\Cubn\A$ and $\Arrn\A$. We now describe one of these and we shall use this one throughout the paper. Given an $n$-cube $A\colon {(2^{n})^{\op}\to\A}$, we see that each edge or one-fold arrow in $A$ is of the form $A_{S\cup\{i\}} \to A_S$ for some $i\in n$ and some subset $S\subset n$ not containing $i$. All edges of this form with the same $i$ are ``parallel'' in the $n$-cube. Thus for each $k$-cube inside $A$, we choose the direction to be that which corresponds to the largest such $i$. As an example, consider the following cube.
\[
\xymatrix@!0@=2.5em{& A_{\{2,0\}} \ar[rr] \ar@{.>}[dd] && A_1 \ar[dd] \\
A_3 \ar[rr] \ar[dd] \ar[ru] && A_2 \ar[dd] \ar[ru]\\
& A_{\{2\}} \ar@{.>}[rr] && A_0\\
A_{\{1,2\}} \ar[rr] \ar@{.>}[ru] && A_{\{1\}} \ar[ru]}
\]
Going from left to right is the direction of ``leaving out $2$'', from front to back is ``leaving out $1$'' and from top to bottom is ``leaving out $0$''. Therefore the right and left square go from front to back, the front, back, top and bottom squares all go from left to right, and the whole cube also goes from left to right. 
\end{convention}

Proposition~\ref{Proposition-DIP-Extension} will show us that remembering the specified directions of an $n$-fold arrow is often not necessary, so that we are mostly happy to use $n$-cubes and $n$-fold arrows synonymously without specifying the isomorphism between them.

\subsection*{Extensions}
We now consider a class of morphisms $\E$ in a category $\A$ satisfying the following axioms:
\begin{enumerate}
\item[(E1)] $\E$ contains all isomorphisms;
\item[(E2)] pullbacks of morphisms in $\E$ exist in $\A$ and are in $\E$;
\item[(E3)] $\E$ is closed under composition.
\end{enumerate}
Given such a class $\E$, we write $\E^{1}$ for the class of arrows $(f_1,f_0)\colon{a\to b}$
\[\vcenter{\xymatrix{A_1 \ar@{}[rd]|{\Rightarrow} \ar[r]^-{f_1} \ar[d]_-{a} & B_1 \ar[d]^-{b}\\A_0 \ar[r]_-{f_0} & B_0}}\]
in $\Arr\A$ such that all arrows in the induced diagram
\[
\vcenter{\xymatrix{A_1 \ar@/^/[rrd]^-{f_1} \ar@/_/[ddr]_-{a} \ar@{.>}[rd] \\
& P \ar@{.>}[r] \ar@{.>}[d] \ar@{}[rd]|<{\pullback} & B_1 \ar[d]^-{b}\\
& A_0 \ar[r]_-{f_0} & B_0}}
\]
are in $\E$. We write $\Ext\A$ for the full subcategory of $\Arr\A$ determined by the arrows in $\E$. 

\begin{remark} 
The pullback in the diagram above exists as we assume that $b$ and~$f_0$ are in $\E$, and (E1) ensures that there is no ambiguity in the choice of pullback.
\end{remark}

\begin{proposition}\label{Proposition-Extensions}
Let $\A$ be a category and $\E$ a class of arrows in $\A$. If~$(\A,\E)$ satisfies (E1)--(E3),
then $(\Ext\A,\E^{1})$ satisfies the same conditions.
\end{proposition}
\begin{proof}
Mutatis mutandis the proof of~\cite[Proposition~3.5]{EGVdL} may be copied.
\end{proof}

\begin{remark}
Pullbacks of double extensions in $\Ext\A$ are computed degree-wise as in $\Arr\A$.
\end{remark}

\begin{remark}
Notice that these axioms have a slightly different appearance to their corresponding ones in~\cite{EverHopf}: there it is important to keep track of the objects which can occur as domains or codomains of extensions. The class of these objects is called $\E^{-}$ and does not occur here because using $(\Ext\A,\E^{1})$ instead of~$(\Arr\A,\E^{1})$ automatically restricts us to the right domains and codomains. In~\cite{EverHopf} this extra care is needed because the construction of the (higher) centralisation functors depends on the categories $\Arrn\A$ being semi-abelian (for~${n\geq 0}$).  Note that, while $\A$ being semi-abelian implies that $\Arr\A$ is semi-abelian, in general~$\Ext\A$ does not keep this property (see, for instance, Section 3 of~\cite{EGVdL}).
\end{remark}

\begin{definition}
If $(\A,\E)$ satisfies (E1)--(E3) then an element of $\E$ is called a \defn{(one-fold) extension} of $\A$ and an element of $\E^1$ a \defn{two-fold extension} or \defn{double extension} of $\A$. We also write $\E^{0}$ for~$\E$. By induction, we obtain a class of arrows $\E^{n}=(\E^{n-1})^{1}$ in $\Arrn\A$ and a full subcategory $\Extn\A$ of $\Arrn\A$ (determined by the elements of $\E^{n-1}$) for all $n\geq 2$. An object of $\Extn\A$ (= an element of $\E^{n-1}$) is called an \defn{$n$-fold extension} of~$\A$. We shall sometimes talk about \defn{$n$-fold $\E$-extensions} or simply of \defn{extensions}.
\end{definition}

\begin{remark}
Notice that, when $\A$ has finite products, (E2) and (E3) imply that the product $f\times g$ of two extensions $f$ and $g$ is also an extension (to see this, observe that $f\times g=(f\times 1)\comp (1\times g)$).
\end{remark}

\begin{example}[Regular epimorphisms]\label{Example-Regular}
If $\A$ is a regular category (finitely complete with coequalisers of kernel pairs and pullback-stable regular epimorphisms, see~\cite{Barr-Grillet-vanOsdol}) and $\E$ is the class of regular epimorphisms in~$\A$, then the pair~$(\A,\E)$ satisfies conditions (E1)--(E3). Indeed, any isomorphism is a regular epimorphism, and regular epimorphisms are pullback-stable and closed under composition. The higher extensions obtained here are the ones considered in~\cite{EGVdL}.
\end{example}

\begin{example}[Projective classes]\label{Example-Projective-Class}
Let $\A$ be a finitely complete category. Recall that a \defn{projective class} on $\A$ is a pair $(\P ,\E)$, where $\P$ is a class of objects of $\A$ and $\E$ a class of morphisms of $\A$, such that $\P$ consists of all $\E$-projectives~$P$, the class $\E$ consists of all $\P$-epimorphisms $f$, and $\A$ has enough $\E$-projectives. 
\[
\xymatrix{& A \ar[d]^-{f} \\
P \ar@{.>}[ru]^-{\exists} \ar[r]_-{\forall} & B}
\]
It is easily seen that $(\A,\E)$ satisfies (E1)--(E3).

Clearly, when $\A$ is a regular category with enough regular projective objects and~$\E$ is the class of regular epimorphisms in $\A$, we regain Example~\ref{Example-Regular}. 

An extreme case is given by taking $\P$ to be the class of all objects of $\A$, so that~$\E$ consists of all split epimorphisms; see also Example~\ref{Example-Naturally-Maltsev}.
\end{example}

\begin{example}[Effective descent morphisms]\label{Example-Effective-Descent}
Let $\A$ be a category with pullbacks, $B$ an object of $\A$ and write $(\A\downarrow B)$ for the slice category over $B$. Given a morphism $p\colon{E\to B}$ in $\A$, let $p^{*}\colon {(\A\downarrow B)\to (\A\downarrow E)}$ be the change of base functor induced by pulling back along $p$. Then $p$ is an \defn{effective descent morphism} if and only if the functor $p^{*}$ is monadic.

Let now $\E$ be the class of effective descent morphisms in $\A$. Then $(\A,\E)$ satisfies (E1)--(E3), and also the axiom~(E4), which we will meet in Section~\ref{Section-Relative-Maltsev}; see e.g.~\cite{Janelidze-Sobral-Tholen}.
\end{example}

\begin{example}[Etale maps]\label{Example-Etale-maps}
Recall that an \'etale map is the same as a local homeomorphism of topological spaces: a continuous map $f\colon{A\to B}$ such that for any element $a\in A$ there is an open neighbourhood $U$ of $a$ such that $f(U)$ is open in $B$ and the restriction of $f$ to a map ${U\to f(U)}$ is a homeomorphism. Taking $\A$ to be the category $\Top$ of topological spaces and $\E$ the class of \'etale maps, it is well known and easily verified that (E1)--(E3) hold for $(\A,\E)$; see for instance~\cite{MacLane-Moerdijk}.
\end{example}

\begin{example}[Topological groups]\label{Example-Topological-Groups}
Let $\Gp\Top$ be the category of topological groups. Since this category is regular, Example~\ref{Example-Regular} implies that $(\Gp\Top,\E)$ satisfies (E1)--(E3) when $\E$ is the class of all regular epimorphisms, which are well known to be the open surjective homomorphisms. 

Another choice of $\E$ would be the class of morphisms which are split as morphisms in the category of topological spaces. It is easy to check that $(\A,\E)$ satisfies the axioms (E1)--(E3). Similarly the category of topological groups together with all morphisms which are split as morphisms of groups satisfies (E1)--(E3). These two examples have been considered important elsewhere in the literature, see for instance \cite{Hochschild-Wigner} and \cite{Wigner-Topological-Groups} (cf.~also \cite[Example 3.3.3]{Tamar_Janelidze}). In fact, both these examples also satisfy the axiom (E4). More examples of this kind are the category of rings together with morphisms which are split in the category of abelian groups $\Ab$ and the category of $R$-modules for a ring $R$, with morphisms which are split in $\Ab$. 
\end{example}

There is an alternative way of looking at extensions which is inspired by~\cite{Brown-Ellis} and~\cite{Donadze-Inassaridze-Porter}.

\begin{proposition}\label{Proposition-DIP-Extension}
Given any $n$-fold arrow $A$, the following are equivalent:
\begin{enumerate}
\item $A$ is an extension;
\item for all $\emptyset\neq I\subseteq n$, the limit $\lim_{J\subsetneq I} A_{J}$ exists and the induced morphism ${A_{I}\to\lim_{J\subsetneq I} A_{J}}$ is in $\E$. 
\end{enumerate}
\end{proposition}
\begin{proof}
We fix an isomorphism between $\Arrn\A$ and $\Cubn\A$, for instance the one described in Convention~\ref{Convention-Iso-Cubes-Arrows}.
For $n\geq 1$ let us denote the class of all $n$-cubes~$A$ in~$\A$ that satisfy Condition~(ii) by $\F^{n-1}$. Note that the classes~$\F^{n}$ may also be defined inductively as follows. The class $\F^{0}$ is $\E$. Now suppose the class~$\F^{n-1}$ is defined. Then $\F^{n}$ consists of all $(n+1)$-cubes~$A$ such that, when considered as an arrow between $n$-cubes in any of the $n+1$ possible directions, the codomain $n$-cube is in $\F^{n-1}$, and moreover the limit $\lim_{J\subsetneq n+1} A_{J}$ exists and the induced morphism ${A_{n+1}\to\lim_{J\subsetneq n+1} A_{J}}$ is in $\E$.

We are to show for all $n$ that $\E^{n}$ consists of all $(n+1)$-fold arrows $A$ of which the corresponding $(n+1)$-cube is in $\F^{n}$. Using the fixed isomorphism between $\Arrnn\A$ and $\Cubnn\A$, we can denote this by $\E^{n}=\F^{n}$. For $n\in\{0,1\}$ we clearly have $\E^{n}=\F^{n}$. Now consider $n\geq 2$ and suppose that $\E^{i-1}=\F^{i-1}$ for all $i\leq n$. Let $A$ be an $(n+1)$-fold arrow which is in $\F^{n}$. Then $A$ is a square in $\Ext^{n-1}\!\A$ as~below.
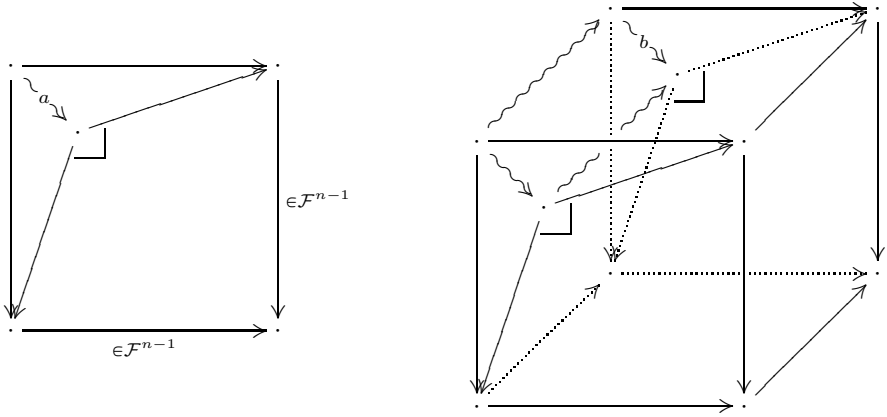
\[
\vcenter{\xymatrix@!0@=2.5em{{\cdot} \ar@{~>}[rd]|-{a} \ar[rrrr] \ar[dddd] &&&& {\cdot} \ar[dddd]^-{\in\F^{n-1}} \\
& {\cdot} \ar@{->}[lddd] \ar@{->}[rrru] \ar@{}[rrrddd]|<{\pullback} \\\\\\
{\cdot} \ar[rrrr]_-{\in\F^{n-1}} &&&& {\cdot}}}
\qquad\qquad
\vcenter{\xymatrix@!0@=2.5em{&& {\cdot} \ar@{~>}[rd]|-{b} \ar[rrrr] \ar@{.>}[dddd]|-{\hole} &&&& {\cdot} \ar[dddd] \\
&&& {\cdot} \ar@{.>}[lddd] \ar@{.>}[rrru] \ar@{}[rd]|<{\pullback} \\
{\cdot} \ar@{~>}[rd] \ar[rrrr] \ar[dddd] \ar@{~>}[rruu] &&&& {\cdot} \ar[dddd] \ar[rruu]\\
& {\cdot} \ar@{~>}[rruu]|-{\hole} \ar@{->}[lddd] \ar@{->}[rrru] \ar@{}[rd]|<{\pullback} \\
&& {\cdot} \ar@{.>}[rrrr] &&&& {\cdot}\\\\
{\cdot} \ar[rrrr] \ar@{.>}[rruu] &&&& {\cdot} \ar[rruu]}}
\]
As $A$ is in $\F^{n}$, both the right and bottom arrow of the square are elements of~$\F^{n-1}=\E^{n-1}$, so their pullback exists. We have to check that the comparison morphism $a$ is also in $\E^{n-1}=\F^{n-1}$. For this we must first check that all possible codomains of the $(n-1)$-cube $a$ are in $\F^{n-2}$. For the direction given in the square this is clear. So consider any other direction, and extend the square to a cube in that direction as on the right hand side above. Then, as~$A$ is an element of~$\F^{n}$, the back square of the cube is in $\F^{n-1}=\E^{n-1}$, so the factorisation~$b$ to the pullback, which is also the chosen codomain of $a$, is in $\E^{n-2}=\F^{n-2}$.

Secondly, it is easy to see that $\lim_{J\subsetneq n+1} A_{J}$ is the same as $\lim_{J\subsetneq n} a_{J}$, since~$a$ is the comparison to a pullback. Therefore $a$ is in $\F^{n-1}=\E^{n-1}$ and $A$ is in~$\E^{n}$.

Conversely, suppose $A$ is in $\E^{n}$. Then $A$ is again a square in $\Ext^{n-1}\!\A$ as before. This time we know that the right and bottom arrows are in the class $\E^{n-1}=\F^{n-1}$, but we must also show it for any other direction. Pick such a direction and extend the square to a cube as before. We are to show that the back square is in~$\F^{n-1}$. We already know this of the right and bottom squares, so the right and bottom arrow of the back square are in $\E^{n-2}$. Now the comparison $a$ is in $\E^{n-1}=\F^{n-1}$, so its codomain $b$ is in $\E^{n-2}$, which shows that the back square is in $\E^{n-1}=\F^{n-1}$. 

Finally, the limits $\lim_{J\subsetneq n+1} A_{J}$ and $\lim_{J\subsetneq n} a_{J}$ are again the same, so as $a$ is in~$\F^{n-1}$, the $(n+1)$-fold arrow $A$ is in $\F^{n}$, as desired.
\end{proof}

\begin{remark}
The condition $I\neq \emptyset$ in (ii) just means that we do not demand~$A_0$ to have global support, that is, we do not demand the unique morphism $A_0\to 1$ to the terminal object $1$ to be an extension. 
\end{remark}

\begin{remark}\label{Remark-DIP}
This proves that, in the case of surjective group homomorphisms (which is an instance of Example~\ref{Example-Regular}), our higher extensions coincide with the \emph{exact cubes} considered in~\cite{Donadze-Inassaridze-Porter}; see also~\cite{Brown-Ellis}.
\end{remark}

Depending on which is more convenient, from now on we shall use either of these characterisations of extensions.

\begin{remark}\label{Remark-Cubes}
Proposition~\ref{Proposition-DIP-Extension} implies that, for an $n$-fold arrow $A$, to be an extension is rightfully a property of the corresponding $n$-cube of $A$. The independence of the chosen isomorphism between $\Arrn\A$ and $\Cubn\A$ means that this property is preserved by all functors ${\Cubn\A\to\Cubn\A}$ induced by an automorphism of~$2^{n}$. Therefore we may sometimes say that an $n$-cube \emph{is} an extension, and the distinction between the different isomorphisms between the categories $\Arrn\A$ and $\Cubn\A$ becomes less important. 
\end{remark}

\section{Resolutions and extensions}\label{Section-Resolutions-and-Extensions}
In this section we analyse the concept of simplicial $\E$-resolution in terms of $n$-fold $\E$-extensions. Our main result in this section is Theorem~\ref{Theorem-Resolution-is-Cube} which states that an augmented semi-simplicial object $\AA$ is an $\E$-resolution if and only if the induced $n$-fold arrows $\arr_{n}\AA$ are $n$-fold extensions for all~$n\geq 1$. From now on we assume that the pair $(\A,\E)$ satisfies (E1)--(E3).

\subsection*{$\E$-resolutions} We start by giving the necessary definitions leading up to that of an $\E$-resolution.

\begin{definition}[Augmented semi-simplicial objects] 
Let $\A$ be a category. The category $\Ss\A$ of \defn{(augmented) semi-simplicial objects} in $\A$ and morphisms between them is the functor category $\Fun ((\Delta_{s})^{\op},\A)$, where $\Delta_{s}$ is the augmented semi-simplicial category. Its objects are the finite ordinals ${n\geq 0}$ and its morphisms are injective order-preserving maps. For a functor
\[
\AA\colon {(\Delta_{s})^{\op}\to \A},
\]
we denote the objects $\AA(n)$ by $A_{n-1}$, and the image of the inclusion ${n\to n+1}$ which does not reach $i$ by $\del_i$, so that an augmented semi-simplicial object $\AA$ corresponds to the following data: a sequence of objects $(A_{n})_{n\geq -1}$ with face operators (or faces) $(\del_{i}\colon {A_{n}\to A_{n-1}})_{0\leq i\leq n}$ for $0\leq n$, 
\[
\xymatrix@=3em{{\cdots} \ar@<1.5ex>[r] \ar@<.5ex>[r] \ar@<-.5ex>[r] \ar@<-1.5ex>[r] & A_2 \ar[r]|-{\del_{1}} \ar@<1ex>[r]^-{\del_{0}} \ar@<-1ex>[r]_-{\del_{2}} & A_1 \ar@<.5ex>[r]^-{\del_{0}} \ar@<-.5ex>[r]_-{\del_{1}} & A_0 \ar[r]^-{\del_{0}} & A_{-1}}
\]
subject to the identities
\[
\del_{i}\comp \del_{j} =\del_{j-1}\comp \del_{i}
\]
for $i<j$. The morphism $\del_{0}\colon{A_{0}\to A_{-1}}$ is called the \defn{augmentation} of $\AA$.
\end{definition}

\begin{definition}[Augmented quasi-simplicial objects] 
In addition to face operators, an \defn{(augmented) quasi-simplicial object} $\AA$ in $\A$ has degeneracy operators (or degeneracies) $(\sigma_{i}\colon {A_{n}\to A_{n+1}})_{0\leq i\leq n}$ for $0\leq n$, subject to the identities
\[
\del_{i}\comp\sigma_{j}=\begin{cases}\sigma_{j-1}\comp \del_{i} & \text{if $i<j$} \\
1 & \text{if $i=j$ or $i=j+1$}\\
\sigma_{j}\comp \del_{i-1} & \text{if $i>j+1$.}\end{cases}
\]
The augmented quasi-simplicial objects in $\A$ with the natural augmented quasi-simplicial morphisms between them form a category $\Sq\A$ which may be seen as a functor category $\Fun ((\Delta_{q})^{\op },\A)$.
\end{definition}

\begin{definition}
If, in addition to the above, an augmented quasi-simplicial object~$\AA$ satisfies 
\[ 
\sigma_{i}\comp \sigma_{j} = \sigma_{j+1}\comp \sigma_{i}
\]
for all $i\leq j$, we recover the usual definition of \defn{(augmented) simplicial object}. Such an object is well known to be a functor \(\AA\colon {\Delta\to \A}\) where \(\Delta\) has finite ordinals as objects and \emph{all} order-preserving maps as morphisms.
\end{definition}

Note that we use Mac\,Lane's notation from~\cite{MacLane} in including the empty set (i.e.\ $0$) in the category $\Delta$, rather than the topologists notation, where $\Delta$ starts at $1$, written as $[0]=\{0\}$. This is the reason for our numbering shift $\AA(n)=A_{n-1}$.

\begin{example}[Comonadic resolutions and Tierney-Vogel resolutions]\label{Example-Comonadic-TV}
Given a comonad $\GG=(G,\epsilon,\delta)$ on a category $\A$, each object $A$ in $\A$ can be extended to an augmented simplicial object $\AA=\GG A$ by setting $A_{-1}=A$ and $A_i=G^{i+1}A$ for $i\geq 0$, with faces $\del_i = G^{i}\epsilon_{G^{n-i}A}\colon {G^{n+1}A\to G^nA}$ and degeneracies $\sigma_i=G^i\delta_{G^{n-i}A}\colon {G^{n+1}A\to G^{n+2}A}$ (see e.g.~\cite{Barr-Beck}). This gives a genuine augmented simplicial object. 

Tierney-Vogel resolutions \cite{Tierney-Vogel, Tierney-Vogel2} of an object on the other hand only give rise to an augmented quasi-simplicial object. Such a resolution is obtained in a category with a projective class $(\P,\E)$ by covering $A=A_{-1}$ by a projective object, then successively taking simplicial kernels (see Definition~\ref{Definition-Simplicial-Kernels} below) and covering these by a projective object again.
\[
\xymatrix@=2em{{\cdots} \ar@<1.5ex>[rr] \ar@<.5ex>[rr] \ar@<-.5ex>[rr] \ar@<-1.5ex>[rr] \ar[dr]_-{\E\ni} && P_2 \ar[rr]^<<{\in\P} \ar@<1ex>[rr]|<{\hole}|<<{\hole}|<<<{\hole} \ar@<-1ex>[rr] \ar[dr]_-{\E\ni} && P_1 \ar@{}[rr]^<<{\in\P} \ar@<.5ex>[rr]|<{\hole}|<<{\hole}|<<<{\hole} \ar@<-.5ex>[rr] \ar[dr]_-{\E\ni} && P_0 \ar[rr]^<<{\in\P}_(.6){\in \E} && A_{-1}\\
& K_3 \ar@<1.5ex>[ur] \ar@<.5ex>[ur] \ar@<-.5ex>[ur] \ar@<-1.5ex>[ur] && K_2 \ar[ur] \ar@<1ex>[ur] \ar@<-1ex>[ur] && K_1 \ar@<.5ex>[ur] \ar@<-.5ex>[ur] & & } 
\]
This does induce degeneracies which commute with the face operators as demanded, but they may not commute with each other as required for a simplicial object.
\end{example}

\begin{definition}[Contractibility]
An augmented semi-simplicial object $\AA$ is \defn{contractible} when there is a sequence of morphisms $(\sigma_{-1}\colon{A_{n-1}\to A_{n}})_{0\leq n}$ such that \[
\del_{0}\comp \sigma_{-1}=1\qquad \text{and} \qquad \del_{i}\comp\sigma_{-1}=\sigma_{-1}\comp\del_{i-1}
\]
for all $i\geq 1$. %%%% want i restricted by n?
\end{definition}

\begin{example}
Given a comonad $\GG$ on a category $\A$, any augmented simplicial object of the form $\GG GA$ is contractible.
\end{example}

\begin{definition}[Simplicial kernels]\label{Definition-Simplicial-Kernels}
Let 
\[
(f_i\colon {X\to Y})_{0\leq i\leq n}
\]
be a sequence of $n+1$ morphisms in the category $\A$. A \defn{simplicial kernel} of $(f_0,\ldots,f_n)$ is a sequence 
\[
(k_i\colon {K\to X})_{0\leq i\leq n+1}
\]
of $n+2$ morphisms in $\A$ satisfying $f_ik_j=f_{j-1}k_i$ for $0\leq i<j\leq n+1$, which is universal with respect to this property. In other words, it is the limit for a certain diagram in $\A$.
\end{definition}

For example, the simplicial kernel of one morphism is just its kernel pair.
When simplicial kernels of a particular augmented semi-simplicial object $\AA$ exist, we can factor $\AA$ through its simplicial kernels as follows.
\[
\xymatrix@=2em{{\cdots} \ar@<1.5ex>[rr] \ar@<.5ex>[rr] \ar@<-.5ex>[rr] \ar@<-1.5ex>[rr] \ar[dr] && A_2 \ar[rr] \ar@<1ex>[rr] \ar@<-1ex>[rr] \ar[dr] && A_1 \ar@<.5ex>[rr] \ar@<-.5ex>[rr] \ar[dr] && A_0 \ar[rr] && A_{-1}\\
& K_3\AA \ar@<1.5ex>[ur] \ar@<.5ex>[ur] \ar@<-.5ex>[ur] \ar@<-1.5ex>[ur] && K_2\AA \ar[ur] \ar@<1ex>[ur] \ar@<-1ex>[ur] && K_1\AA \ar@<.5ex>[ur] \ar@<-.5ex>[ur] & & }
\]
Here the $K_{n+1}\AA$ are the simplicial kernels of the morphisms $(\del_i)_{i}\colon {A_{n} \to A_{n-1}}$. We may also sometimes write $K_0\AA=A_{-1}$.

%\begin{remark}
%If $\A$ has all pullbacks then simplicial kernels exist, as they may be computed as a succession of pullbacks (cf.\ the proof of Proposition~\ref{Proposition-Contractible-plus-Kan-is-Resolution}).
%\end{remark}

\begin{definition}
If all faces $\del_i$ of an (augmented) semi-simplicial object $\AA$ are in~$\E$, we call $\AA$ an \defn{(augmented) $\E$-semi-simplicial object}.
\end{definition}

\begin{definition}\label{Definition-Resolution}
An (augmented) semi-simplicial object $\AA$ is said to be \defn{$\E$-exact at $A_{n-1}$} when the simplicial kernel $K_{n}\AA$ exists and the factorisation ${A_{n}\to K_{n}\AA}$ is in $\E$.

An augmented semi-simplicial object $\AA$ is called an \defn{$\E$-resolution (of $A_{-1}$)} when~$\AA$ is $\E$-exact at $A_{n}$ for all $n\geq -1$.
\end{definition}

\begin{remark}
An $\E$-resolution is always an augmented $\E$-semi-simplicial object. Notice that since $K_0\AA=A_{-1}$, $\AA$ is $\E$-exact at $A_{-1}$ just when $\del_0\colon {A_0\to A_{-1}}$ is in $\E$.
\end{remark}

\begin{notation}\label{Notation-A^{-}}
Let $\AA$ be an augmented semi-simplicial object in $\A$. We can form another augmented semi-simplicial object $\AA^-$ by setting
\[
A^-_{n-1}=A_{n}\qquad\text{and}\qquad\del^-_i=\del_{i+1}\colon {A_{n+1}\to A_{n}},
\]
for $n\geq 0$ and $0\leq i\leq n$. This is the augmented semi-simplicial object obtained from $\AA$ by leaving out $A_{-1}$ and all $\del_0\colon {A_n\to A_{n-1}}$. Observe that $\del=(\del_{0})_n$ defines a morphism from $\AA^-$ to~$\AA$. 

When $\AA$ is a (quasi)-simplicial object, the degeneracy operators can be shifted in the same way to give a (quasi)-simplicial object $\AA^{-}$ and a morphism $\del\colon{\AA^{-}\to \AA}$ of (quasi)-simplicial objects.
\end{notation}

\begin{remark}\label{Remark-Contractible}
Note that $\AA^{-}$ is contractible when $\AA$ is an augmented quasi-simplicial object, and that an augmented semi-simplicial object $\AA$ is contractible if and only if $\del\colon{\AA^{-}\to \AA}$ is a split epimorphism of augmented semi-simplicial objects.
\end{remark}

\begin{remark}\label{Remark-Simp-Object-of-Arrows}
We may also view the morphism $\del\colon{\AA^{-}\to \AA}$ as an augmented semi-simplicial object of arrows, say $\BB$, with the $\del_0\colon {A_{n+1}\to A_n}$ forming the objects $B_n$. Notice that, when we view $\del$ as a morphism of semi-simplicial objects in $\A$, the direction of a square goes parallel to the $\del_0$ as in the left diagram below, depicting the morphism ${\del^-_i\to\del_i}$, whereas if we view it as a semi-simplicial object of arrows, the direction goes from one $\del_0$ to the next as in the right diagram, displaying the~$\del_i$ of the semi-simplicial object $\BB$.
\[
\vcenter{\xymatrix@!0@=10ex{A_{n+1} \ar[d]_-{\del_0} \ar[r]^-{\del_{i+1}=\del^-_i} \ar@{}[dr]|-{\Downarrow} & A_n \ar[d]^-{\del_0} \\
A_n \ar[r]_-{\del_i} & A_{n-1}}}
\qquad\qquad
\vcenter{\xymatrix@!0@=10ex{A_{n+1} \ar[d]_-{\del_0=B_n} \ar[r]^-{\del_{i+1}} \ar@{}[dr]|-{\Rightarrow} & A_n \ar[d]^-{\del_0=B_{n-1}}\\
A_n \ar[r]_-{\del_i} & A_{n-1}}}
\]
\end{remark}

\subsection*{Truncations and higher arrows}
If an augmented semi-simplicial object $\AA$ in~$\A$ is truncated at level~$n$, it corresponds to an $(n+1)$-fold arrow in $\A$ as follows. Truncation at level zero automatically gives a morphism $\del_0\colon {A_0\to A_{-1}}$. When we truncate at level one, we can use $\del\colon{\AA^{-}\to \AA}$ to view all the remaining information as an augmented semi-simplicial object of morphisms~$\BB$, truncated at level zero.
\[
\xymatrix@!0@=10ex{A_1 \ar@{}[rd]|-{\Rightarrow} \ar[d]_-{\del_0=B_0} \ar[r]^-{\del_1} & A_0 \ar[d]^-{\del_0=B_{-1}} \\
A_0 \ar[r]_-{\del_0} & A_{-1}}
\]
This can clearly be viewed as a double arrow. Similarly an augmented semi-simplicial object truncated at level $n$ corresponds to an augmented semi-simpli\-cial object of morphisms truncated at level $n-1$, which by induction corresponds to an $n$-fold arrow of morphisms in~$\A$, which in turn can be viewed as an $(n+1)$-fold arrow of objects in $\A$.

\begin{definition}\label{Definition-Cube}
The above determines a functor $\arr_{n}\colon{\Ss\A\to \Arrn\A}$ for any ${n\geq 1}$. We also consider
\[
\arr_{0}\colon{\Ss\A\to \Arr^{0}\!\A=\A\colon \AA\mapsto A_{-1}.}
\]
\end{definition}

This description may be illustrated by the commutative square
\begin{equation}\label{Truncation}
\vcenter{\xymatrix@1@C=3em@R=2em{\Ss\A \ar[r]^-{\arr_{n+1}} \ar[d] & \Arr^{n+1}\!\A \ar[d] \\
\Ss\Arr\A \ar[r]_-{\arr_{n}} & \Arr^{n}\Arr\A}}
\end{equation}
in which the left downward arrow sends $\AA$ to $\del\colon{\AA^{-}\to \AA}$ viewed as an augmented semi-simplicial object of arrows $\BB$ as in Remark~\ref{Remark-Simp-Object-of-Arrows}. The right downward arrow is the following isomorphism. We know that $\Arrnn\A\iso \A^{(2^{n+1})^{\op}}$ and $\Arr^{n}\Arr\A\iso \A^{(2^{n})^{\op}\times 2^{\op}}$ (using the isomorphism fixed in Convention~\ref{Convention-Iso-Cubes-Arrows}), so it is enough to describe the isomorphism between $2^{n+1}$ and $2^{n}\times 2$. Given a set $S\subset n$, we write~$S^{+1}$ for the set obtained from $S$ by shifting all elements up by one, that is, we have $i\in S$ if and only if $i+1\in S^{+1}$. Using this notation, we choose the isomorphism which sends a couple $(S,0)\in 2^{n}\times2$ to $S^{+1}\in 2^{n+1}$ and~$(S,1)\in 2^{n}\times2$ to $S^{+1}\cup\{0\}\in 2^{n+1}$.

There is another way of obtaining the functors $\arr_{n}$ which may be described as follows. For any $n\geq 0$, let
\[
F_{n}\colon{2^{n}\to \Delta^{+}_{s}}
\]
be the functor which maps a set $S\subseteq n$ to the associated ordinal $|S|$, and an inclusion $S\subseteq T$ to the corresponding order-preserving map ${|S|\to |T|}$: if
\[
T=\{x_{0}<x_{1}<\cdots<x_{|T|-1}\}
\]
and $S=\{x_{i_{0}}<x_{i_{1}}<\cdots<x_{i_{|S|-1}}\}$ then the map $|S|\to |T|$ sends $k$ to $i_{k}$. We again fix the isomorphism $\Arrn\A\iso \Cubn\A$ as described in Convention~\ref{Convention-Iso-Cubes-Arrows}.

\begin{lemma}\label{Lemma-Two-arrs}
For any $n\geq 0$, the functor $\arr_{n}\colon{\Ss\A\to \Arrn\A}$ is equal to
\[
\Fun(F_{n}^{\op},-)\colon{\Fun((\Delta^{+}_{s})^{\op},\A)\to \Fun((2^{n})^{\op},\A)}.
\]
\end{lemma}
\begin{proof}
As $\arr_{n+1}$ is defined inductively by the square~\eqref{Truncation} above and $\arr_{0}$ clearly coincides with $\Fun(F_{0}^{\op},-)$, it is enough to check that the square
\[
\xymatrix@1@C=5em@R=2em{\Ss\A \ar[r]^-{\Fun(F_{n+1}^{\op},-)} \ar[d] & \Arr^{n+1}\!\A \ar[d] \\
\Ss\Arr\A \ar[r]_-{\Fun(F_{n}^{\op},-)} & \Arr^{n}\Arr\A}
\]
commutes.
\end{proof}

\begin{lemma}\label{Lemma-Codomains}
Let $\AA$ be an augmented semi-simplicial object, $n\geq 1$ and $\arr_{n}\AA$ the induced $n$-fold arrow. As mentioned above, the corresponding $n$-cube may be considered as an arrow between $(n-1)$-cubes in $n$ different ways. The codomains of all of these arrows determine the same $(n-1)$-cube.
\end{lemma}
\begin{proof}
A subset $S$ of $n$ determines the full subcategory $2^{S}$ of $2^{n}$. If $|S|=n-1$, the restriction of $\arr_{n}\AA$ to $2^{S}$ is one of the codomains considered in the statement of the lemma. Given two subsets $S$ and $T$ of $n$ such that $|S|=|T|=n-1$, the subcategories $2^{S}$ and $2^{T}$ are mapped by the functor $F_{n}$ to one and the same subcategory of $\Delta^{+}_{s}$. Thus, using the alternative description of the functor $\arr_{n-1}$ from Lemma~\ref{Lemma-Two-arrs}, we see that for any augmented semi-simplicial object~$\AA$, the two induced restrictions of $\arr_{n}\A$ to the $(n-1)$-cubes determined by $S$ and $T$ are equal to each other.
\end{proof}

This brings us to the main result of this section.
\begin{theorem}\label{Theorem-Resolution-is-Cube}
An augmented semi-simplicial object $\AA$ is an $\E$-resolution if and only if $\arr_{n}\AA$ is an $n$-fold extension for all $n\geq 1$.
\end{theorem}
\begin{proof}
If $\AA$ is an $\E$-resolution, then $\arr_1\AA=\del_0\colon {A_0\to A_{-1}}$ is an extension by definition, and conversely $\arr_1\AA$ being an extension implies that $\AA$ is $\E$-exact at~$A_{-1}$, which is the first condition for $\AA$ to be an $\E$-resolution. For ${n\geq 2}$, consider the full subcategory $\D$ of~$2^{n}$ determined by all sets $S\subseteq n$ with $n-2\leq |S|\leq n-1$. It is easy to see that~$\D$ is initial in the full subcategory of $(2^{n})^{\op}$ containing all objects except~$n$. It follows that, for any $n$-fold arrow $A\colon{(2^{n})^{\op}\to\A}$,
\[
\lim_{J\subsetneq n}A_{J}=\lim_{J\in |\D|}A_{J}.
\]
If now $A=\arr_{n}\AA$ for an augmented semi-simplicial object $\AA$, the subdiagram is exactly the diagram which determines $K_{n}\AA$. 
\[
\vcenter{\xymatrix@!0@=3em{& A_{1} \ar[rr]^-{\del_{1}} \ar@{->}[dd]^(.25){\del_{0}}|-{\hole} && A_{0} \ar@{.>}[dd] \\
L \ar@{-->}[rr]_(.25){l_{2}} \ar@{-->}[dd]_-{l_{0}} \ar@{-->}[ru]^-{l_{1}} && A_{1} \ar@{->}[dd]^(.25){\del_{0}} \ar[ru]_-{\del_{1}}\\
& A_{0} \ar@{.>}[rr] && A_{-1}\\
A_{1} \ar[rr]_-{\del_{1}} \ar@{->}[ru]_-{\del_{0}} && A_{0} \ar@{.>}[ru]}}
\]
This shows that if either limit exists, then the other exists and they are the same. This automatically proves one of the implications, using the condition for extensions given in (ii) of Proposition~\ref{Proposition-DIP-Extension}. For the other, we must also show that each codomain of $\arr_{n}\AA$ is an $(n-1)$-extension. Lemma~\ref{Lemma-Codomains} shows that checking one codomain suffices. The canonical codomain of the $n$-fold arrow~$\arr_{n}\AA$ is $\arr_{n-1}\AA$ and as such is an extension by induction.
\end{proof}

This makes clear that we can view (semi)-simplicial resolutions as ``infinite dimensional extensions'', and a higher extension as a finite-dimensional resolution.

\begin{remark}
Theorem \ref{Theorem-Resolution-is-Cube} shows, in particular, that one can use the $n$-trunca\-tion of a canonical simplicial resolution $\GG A$ of an object $A$ as an $n$-fold projective presentation of $A$ (a special kind of higher extension) in order to compute the higher Hopf formulae which give the homology of $A$ (as e.g.\ in the article~\cite{EGVdL}).
\end{remark}

\begin{corollary}\label{Corollary-Minus}
An augmented semi-simplicial object $\AA$ is an~$\E$-resolution if and only if the augmented semi-simplicial object of arrows $\del\colon{\AA^{-}\to \AA}$ is an~$\E^{1}$-resolu\-tion.
\end{corollary}
\begin{proof}
For $n\geq 0$, the $(n+1)$-truncation $\arr_{n+1}\AA$ of the augmented semi-simplicial object of objects $\AA$ is an $(n+1)$-fold \E-extension in~$\A$ precisely when the $n$-truncation $\arr_{n}\BB$ of the augmented semi-simplicial object of arrows $\BB=\del\colon{\AA^{-}\to \AA}$ is an $n$-fold $\E^1$-extension in~$\Ext\A$. 
\end{proof}

\section{The relative Mal'tsev axiom}\label{Section-Relative-Maltsev}
We now investigate a relative version of the Kan property for simplicial objects and its connections to properties of higher extensions. The main condition on higher extensions in this context is a relative Mal'tsev axiom, which is equivalent to two other important conditions.

Throughout this section, we consider the following axioms on $(\A,\E)$:
\begin{enumerate}
\item[(E1)] $\E$ contains all isomorphisms;
\item[(E2)] pullbacks of morphisms in $\E$ exist in $\A$ and are in $\E$;
\item[(E3)] $\E$ is closed under composition;
\item[(E4)] if $g\comp f\in \E$ then $g\in\E$;
\item[(E5)] the \defn{$\E$-Mal'tsev axiom}: any split epimorphism of extensions
\[
\xymatrix{A_1 \ar@<.5ex>[r]^-{f_1} \ar[d]_-{a} & B_1 \ar[d]^-{b} \ar@<.5ex>[l]\\
A_0 \ar@<.5ex>[r]^-{f_0} & B_0 \ar@<.5ex>[l]}
\]
in $\A$ is a double extension.
\end{enumerate}
Notice that (E1) and (E4) together imply that all split epimorphisms are in~$\E$.

\begin{remark}
Axioms (E1)--(E4) say exactly that our class $\E$ generates a Grothendieck topology (or Grothendieck coverage) on $\A$, see e.g.~\cite[Defini\-tion~C2.1.8]{Johnstone:Elephant}.
\end{remark}

\subsection*{Axiom (E5)}
We will first show that (E5) is equivalent to two other conditions, connecting the class of extensions $\E$ and the corresponding class $\E^1$ of double extensions. To do this, we make use of the following lemma.

\begin{lemma}\label{Extension-Left-Right}
Let $(\A,\E)$ satisfy (E1)--(E4). Consider a diagram
\[
\vcenter{\xymatrix{\R[f_1] \ar[d]_{r} \ar@<-.5ex>[r]_-{\pi_{1}} \ar@<.5ex>[r]^-{\pi_{0}} & A_1 \ar[r]^{f_1} \ar[d]_{a} & B_1 \ar[d]^{b}\\
\R[f_0] \ar@<-.5ex>[r]_-{\pi_{1}'} \ar@<.5ex>[r]^-{\pi_{0}'} & A_0 \ar[r]_{f_0} & B_0}}
\]
in $\A$ with $a$, $b$, $f_1$ and~$f_0$ in $\E$ and $\R[f_1]$ and $\R[f_0]$ the kernel pairs of~$f_1$ and~$f_0$. Either (hence both) of the left hand side commutative squares is in~$\E^{1}$ if and only if the right hand side square is in $\E^{1}$.
\end{lemma}
\begin{proof}
The right-to-left implication follows from~(E2) for the class~$\E^{1}$. Now suppose that the square $a\comp \pi_{0}=\pi_{0}'\comp r$ is in $\E^{1}$. The diagram induces the following commutative cube and the right hand side commutative comparison square to the pullback.
\[
\vcenter{\xymatrix@!0@=3em{& A_1 \ar[rr]^-{f_1} \ar@{.>}[dd] && B_1 \ar[dd] \\
\R[f_1] \ar[rr] \ar[dd]_-{r} \ar[ru] && A_1 \ar[dd] \ar[ru]_-{f_1}\\
& A_0 \ar@{.>}[rr]^(.75){f_0} && B_0\\
\R[f_0] \ar[rr] \ar@{.>}[ru] && A_0 \ar[ru]_-{f_0}}}
\qquad\qquad
\vcenter{\xymatrix@!0@R=4em@C=8em{\R[f_1] \ar[r]^-{\pi_{1}} \ar[d]_{\langle r,\pi_{0}\rangle} & A_1 \ar[d]^-{\langle a,f_1\rangle } \\
\R[f_0]\times_{A_0}A_1 \ar[r]_-{\pi_{1}\times_{f_0}f_1} & A_0\times_{B_0}B_1}}
\]
In the square, the morphism $\langle r,\pi_{0}\rangle$ is in $\E$ by assumption. Furthermore, the morphism $\pi_{1}\times_{f_0}f_1$ is in $\E$ as a pullback of the extension $f_1$. It follows by (E3) that $\langle a,f_1\rangle\comp \pi_{1}$ is an extension, and so (E4) implies that $\langle a,f_1\rangle$ is in $\E$.
\end{proof} 

\begin{proposition}\label{Proposition-(E5)}
Let $\A$ be a category and $\E$ a class of arrows in $\A$ which satisfies the axioms (E1)--(E4). The following conditions are equivalent:
\begin{enumerate}
\item (E4) holds for the class $\E^{1}$, that is, if $g\comp f\in \E^1$ then $g\in \E^1$;
\item (E5) holds, that is, all split epimorphisms of extensions are in $\E^1$;
\item every split epimorphism of split epimorphisms, i.e., every diagram
\begin{equation}\label{Double-Split-Epi}
\vcenter{\xymatrix@=3em{A_{1} \ar@<-.5ex>[r]_-{f_{1}} \ar@<-.5ex>[d]_-{a} & B_1 \ar@<-.5ex>[d]_-{b} \ar@<-.5ex>[l]_-{\overline{f_{1}}}\\
A_{0} \ar@<-.5ex>[u]_-{\overline{a}} \ar@<-.5ex>[r]_-{f_{0}} & B_{0}, \ar@<-.5ex>[u]_-{\overline{b}} \ar@<-.5ex>[l]_-{\overline{f_{0}}}}}
\end{equation}
such that $f_{0}a =b f_{1}$, $\overline{f_{0}} b =a \overline{f_{1}}$, $\overline{b} f_{0} = f_{1}\overline{a}$, $\overline{a}\overline{f_{0}}=\overline{f_{1}}\overline{b}$ and $f_{0} \overline{f_{0}}=1_{B_{0}}$, $f_{1} \overline{f_{1}}=1_{B_{1}}$, $a \overline{a}=1_{A_{0}}$, $b \overline{b}=1_{B_{0}}$ is a double extension;
\item consider the diagram
\begin{equation}\label{LeftRight}
\vcenter{\xymatrix{\R[f_1] \ar[d]_{r} \ar@<.5ex>[r] \ar@<-.5ex>[r] & A_1 \ar[r]^{f_1} \ar[d]_{a} & B_1 \ar[d]^{b}\\
\R[f_0] \ar@<.5ex>[r] \ar@<-.5ex>[r] & A_0 \ar[r]_{f_0} & B_0}}
\end{equation}
in $\A$ with $a$, $b$, $f_1$ and $f_0$ in $\E$; the arrow $r$ is in $\E$ if and only if the right hand side square is in $\E^{1}$.
\end{enumerate}
\end{proposition}
\begin{proof}
Since all isomorphisms of extensions are double extensions, we see that (i) implies (ii). Clearly (iii) is a special case of (ii). Now suppose that (iii) holds and consider a diagram \eqref{LeftRight} as in (iv). Lemma~\ref{Extension-Left-Right} automatically gives one direction, that is, if the right hand square is a double extension, then~$r$ is in~$\E$. Conversely, taking kernel pairs vertically of the left hand side square gives us a square as in~(iii). By assumption this square is a double extension. Using Lemma~\ref{Extension-Left-Right} twice we see that all squares in the diagram are double extensions.

Finally, suppose that (iv) holds and consider the morphisms~$f$ and~$g$ in $\Ext\A$ as in the diagram below.
\[
\xymatrix{\R[a] \ar[r]^-{r} \ar@<-.5ex>[d] \ar@<.5ex>[d] & \R[b] \ar[r]^-{s} \ar@<-.5ex>[d] \ar@<.5ex>[d] & \R[c] \ar@<-.5ex>[d] \ar@<.5ex>[d]\\
A_1 \ar[r]^-{f_1} \ar[d]_-{a} & B_1 \ar[r]^-{g_1} \ar[d]^-{b} & C_1 \ar[d]^-{c}\\
A_0 \ar[r]_-{f_0} & B_0 \ar[r]_-{g_0} & C_0}
\]
Assume that the composite $g\comp f$ is a double extension. Then by assumption $s\comp r$ is in $\E$. Axiom (E4) implies that $s$ is in $\E$, so~(iv) implies that $g$ is in $\E^{1}$.
\end{proof}

The axioms (E1)--(E5) ``go up'':

\begin{proposition}\label{(E1)--(E5)-go-up}
Let $\A$ be a category and $\E$ a class of arrows in $\A$. If~$(\A,\E)$ satisfies (E1)--(E5) then $(\Ext\A,\E^{1})$ satisfies the same conditions.
\end{proposition}
\begin{proof}
By Proposition~\ref{Proposition-(E5)}, Axiom (E5) for $(\A,\E)$ is equivalent to Axiom (E4) for~$(\Ext\A,\E^{1})$. To see that Axiom (E5) for $(\Ext\A,\E^{1})$ holds, consider a split epimorphism of double extensions in $\A$ such as the following left hand side cube and recall Remark~\ref{Remark-Cubes}.
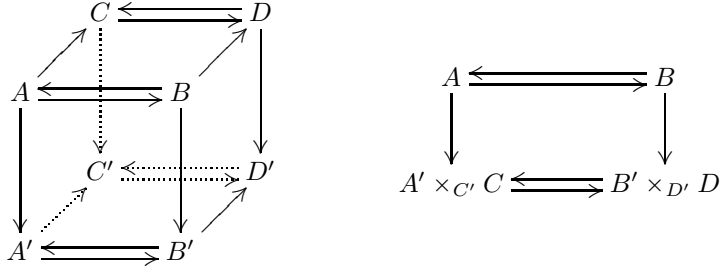
\[
\vcenter{\xymatrix@!0@=3em{& C \ar@<-.5ex>[rr] \ar@{.>}[dd] && D \ar[dd] \ar@<-.5ex>[ll] \\
A \ar@<-.5ex>[rr] \ar[dd] \ar[ru] && B \ar[dd] \ar[ru] \ar@<-.5ex>[ll]\\
& C' \ar@<-.5ex>@{.>}[rr] && D' \ar@{.>}@<-.5ex>[ll]\\
A' \ar@<-.5ex>[rr] \ar@{.>}[ru] && B' \ar[ru] \ar@<-.5ex>[ll]}}
\qquad\qquad
\vcenter{\xymatrix@!0@R=4em@C=8em{A \ar@<-.5ex>[r] \ar[d] & B \ar[d] \ar@<-.5ex>[l]\\
A'\times_{C'}C \ar@<-.5ex>[r] & B'\times_{D'}D \ar@<-.5ex>[l]}}
\]
The arrows pointing to the right are split epimorphisms. By assumption, the cube's left and right hand side squares are double extensions; Axiom (E5) for~$(\A,\E)$ implies, moreover, that the front, back, top and bottom squares are also double extensions. Hence the induced right hand side comparison square to the pullback exists. It is a double extension by Axiom (E5) for~$(\A,\E)$.
\end{proof}

\subsection*{The $\E$-Kan property}
The Kan property is well known for simplicial sets and simplicial groups~\cite{Moore:Kan} and was used in~\cite{Carboni-Kelly-Pedicchio} to extend the characterisation of the Mal'tsev property in terms of simplicial objects from varieties to regular categories (proving a conjecture of M.~Barr~\cite{Barr-Grillet-vanOsdol}). We slightly adapt the definition to obtain a relative notion of $\E$-Kan simplicial objects.

\begin{definition}\label{Definition-Kan}
Let $\AA$ be a semi-simplicial object and consider $n\geq 2$ and $0\leq k\leq n$. The \defn{object of $(n,k)$-horns in $\AA$} is an object $A(n,k)$ together with arrows $a_{i}\colon{A(n,k)\to A_{n-1}}$ for $i\in \{0,\dots, n\}\setminus\{k\}$ satisfying
\[
\text{$\del_{i}\comp a_{j}=\del_{j-1}\comp  a_{i}$ for all $i<j$ with $i$, $j\neq k$}
\]
which is universal with respect to this property. We also define $A(1,0)=A(1,1)=A_{0}$.

A semi-simplicial object is \defn{$\E$-Kan} when all $A(n,k)$ exist and all comparison morphisms ${A_{n}\to A(n,k)}$ are in $\E$. In particular, the comparison morphisms to the $(1,k)$-horns are $\del_0\colon{A_1\to A(1,0)}$ and $\del_1\colon {A_1\to A(1,1)}$.
\end{definition}
%
%\begin{remark}
%If $\AA$ is a semi-simplicial object in a category $\A$ with all pullbacks, then all $A(n,k)$ exist. 
%\end{remark}

For simplicity, we assume that $\A$ has a terminal object so that every semi-simplicial object has a canonical augmentation. In fact this augmentation is only needed to allow a formulation in terms of cubes.

\begin{proposition}\label{Kan-as-Extension}
Let $\AA$ be a semi-simplicial object and $\arr_{n+1}\AA$ the $(n+1)$-cube induced by (the canonical augmentation of) $\AA$ for some $n\geq 1$. Then~$\AA$ satisfies the $\E$-Kan property at level $n$ (i.e., for all $(n,k)$-horns) if and only if the domains of all arrows of $n$-cubes in $\arr_{n+1}\AA$ (i.e., in all possible directions) are extensions.
\end{proposition}
\begin{proof}
A domain of any arrow of $n$-cubes in $\arr_{n+1}\AA$ is given by the $n$-subcube which involves all faces $\del_i\colon {A_n\to A_{n-1}}$ except for one particular $\del_k$. In the same way as in the proof of Theorem~\ref{Theorem-Resolution-is-Cube}, we see that the limit of the subdiagram of this $n$-cube without the initial object $A_n$ is exactly the $(n,k)$-horn object. Therefore, by induction on $n$, the $\E$-Kan property holds for the $(n,k)$-horn object if and only if that particular cube is an extension.
\end{proof}

Using Theorem~\ref{Theorem-Resolution-is-Cube} this gives us in particular

\begin{corollary}\label{Corollary-A^{-}-is-E-resolution}
Let $(\A,\E)$ satisfy (E1)--(E4). For any $\E$-semi-simplicial object~$\AA$ which is $\E$-Kan, the associated augmented semi-simplicial object~$\AA^{-}$ is an $\E$-resolu\-tion.\noproof
\end{corollary}

As a first illustration of what the relative Kan property is useful for, we show that a contractible augmented $\E$-semi-simplicial object $\AA$ which is also $\E$-Kan is always an $\E$-resolution. For this we make an observation about the existence of simplicial kernels. 

\begin{lemma}\label{Lemma-Simplicial-Kernel}
If $\AA$ is a resolution up to level $n$ then $K_{n+1}\AA$ exists.
\end{lemma}
\begin{proof}
This follows from Lemma~\ref{Lemma-Codomains} and the following property of higher cubes, which is proved inductively as in Proposition~\ref{Proposition-DIP-Extension}: if all codomains in an $(n+2)$-cube $A$ are extensions, then the limit $\lim_{J\subsetneq n+2}A_{J}$ exists.
\end{proof}

\begin{proposition}\label{Proposition-Contractible-plus-Kan-is-Resolution}
Let $(\A,\E)$ satisfy (E1)--(E4). An augmented $\E$-semi-simpli\-cial object in $(\A,\E)$ which is contractible and satisfies the $\E$-Kan property is an~$\E$-resolution.
\end{proposition}
\begin{proof} 
As $\AA$ is an $\E$-semi-simplicial object, in particular the morphism
\[
\del_0\colon {A_0\to A_{-1}=K_0\AA}
\]
is in $\E$, so $\AA$ is an $\E$-resolution at level 0.

Now let $\AA$ be a resolution up to level $n$.
By Lemma~\ref{Lemma-Simplicial-Kernel}, we can assume inductively that the simplicial kernel $K_{n+1}\AA$ exists.  So we can consider the diagram
\[
\xymatrix@=3em{A_{n+2} \ar[d]_-{\del_{0}} \ar[r] & A(n+2,0) \ar@<1.5ex>[r]^-{ a_{1}} \ar@<-1.5ex>[r]_-{ a_{n+2}}^-{\vdots} \ar@<-.5ex>[d]_-{r}  & A_{n+1} \ar@<1.5ex>[r]^-{\del_{1}} \ar@<-1.5ex>[r]_-{\del_{n+1}}^-{\vdots} \ar@<-.5ex>[d]_-{\del_{0}} & A_{n} \ar@<-.5ex>[d]_-{\del_{0}}\\
A_{n+1} \ar[r]_-{\langle\del_{0},\ldots,\del_{n+1}\rangle} & K_{n+1}\AA \ar@<1.5ex>[r]^-{k_{0}} \ar@<-.5ex>[u] \ar@<-1.5ex>[r]_-{k_{n+1}}^-{\vdots} & A_{n} \ar@<1.5ex>[r]^-{\del_{0}} \ar@<-1.5ex>[r]_-{\del_{n}}^-{\vdots}\ar@<-.5ex>[u]_-{\sigma_{-1}} & A_{n-1} \ar@<-.5ex>[u]_-{\sigma_{-1}} }
\] %%%% let op met de index -1 in het diagram????
in which $\K_{n+1}\AA$ and $A(n+2,0)$ are the simplicial kernels of the given morphisms. As $\del\colon{\AA^{-}\to \AA}$ is a split epimorphism of augmented $\E$-semi-simplicial objects by Remark~\ref{Remark-Contractible}, the induced morphism $r$ between the limits is split epic, and thus an extension. In fact, $r=K_{n}\BB$, since $\BB$ is a resolution up to level $n-1$, so this simplicial kernel can be constructed by going to cubes as in Lemma~\ref{Lemma-Simplicial-Kernel}. It is pointwise because pullbacks of double extensions are pointwise in $\Ext \A$. The comparison morphism ${A_{n+2}\to A(n+2,0)}$ is an extension as $\AA$ is $\E$-Kan, so the composite $\del_0\comp\langle\del_{0},\ldots,\del_{n+1}\rangle$ is an extension by (E3). Therefore $\langle\del_{0},\ldots,\del_{n+1}\rangle$ is an extension by the cancellation property~(E4).
\end{proof}

Now we prove that, with a small extra assumption, the relative Mal'tsev axiom~(E5) is equivalent to \emph{every} \emph{quasi}-simplicial object being $\E$-Kan. 

\begin{remark}
Notice that any (quasi)-simplicial object is automatically an~$\E$-semi-simplicial object, as all split epimorphisms are in $\E$. However, this does not automatically extend to \emph{augmented} (quasi)-simplicial objects.
\end{remark}

\begin{proposition}\label{Proposition-Relative-Maltsev-Converse}
If $(\A,\E)$ satisfies (E1)--(E5) then every quasi-simpli\-cial object in $\A$ satisfies the $\E$-Kan property.
\end{proposition}
\begin{proof}
For every quasi-simplicial object $\AA$, the $\E$-Kan property for $A(1,k)$ just says that $\del_{0}$ and $\del_{1}\colon{A_{1}\to A_{0}}$ are in $\E$, which is automatically satisfied thanks to (E1) and (E4), which imply that all split epimorphisms are in $\E$.

Now assume that the $\E$-Kan property holds up to level $n$ for all~$(\A,\E)$ which satisfy (E1)--(E5). Let $\AA$ be a quasi-simplicial object in~$\A$ and write~$\BB=\del\colon{\AA^{-}\to\AA}$ for the induced quasi-simplicial object in $\Ext\A$. Axiom~(E5) for $(\A,\E)$ ensures that $(\Ext\A,\E^1)$ also satisfies (E4) and (E5) (Proposition~\ref{(E1)--(E5)-go-up}). So by assumption, $\BB$ is $\E^{1}$-Kan up to level~$n$. By Proposition~\ref{Kan-as-Extension} this means that the domains of the $(n+1)$-cube $\arr_{n+1}\BB$ in $\Ext\A$ are $n$-fold $\E^1$-extensions. Hence in the $(n+2)$-cube $\arr_{n+2}\AA$ in $\A$, certain domains are $(n+1)$-fold \E-extensions. This almost shows that~$\AA$ is $\E$-Kan at level $n+1$: the property holds for all domains but one. The missing case follows by symmetry.
\end{proof}

We can also prove a converse of Proposition~\ref{Proposition-Relative-Maltsev-Converse}, however we now need $\A$ to have all simplicial kernels so that truncation of simplicial objects has a right adjoint.

\begin{proposition}\label{Proposition-Relative-Maltsev}
Let $\A$ be a category with simplicial kernels and $\E$ a class of morphisms in $\A$ which satisfies (E1)--(E4). If every simplicial object in $\A$ has the $\E$-Kan property then $(\A,\E)$ satisfies (E5).
\end{proposition}
\begin{proof}
We have to prove that every split epimorphism of split epimorphisms in $\A$ is a double extension. We may reduce the situation to a (truncated) contractible augmented $\E$-simplicial object
\addtocounter{equation}{2}
\begin{equation}\label{Trunc}
\xymatrix{A_{1} \ar@<-2.25ex>[r]_-{\del_{1}} \ar@<.75ex>[r]|-{\del_{0}} & A_{0} \ar@<-.75ex>[r]_-{\del_{0}} \ar@<.75ex>[l]|-{\sigma_{0}} \ar@<-2.25ex>[l]_-{\sigma_{-1}} & A_{-1}. \ar@<-.75ex>[l]_-{\sigma_{-1}}}
\end{equation}
Consider the following split epimorphism of split epimorphisms (any of the four possible squares commutes, and the arrows pointing down or right are the split epimorphisms).
\begin{equation}\label{Split-Epi-of-Split-Epis}
\vcenter{\xymatrix@=3em{A \ar@<-.5ex>[r]_-{f} \ar@<-.5ex>[d]_-{a} & B \ar@<-.5ex>[d]_-{b} \ar@<-.5ex>[l]_-{\overline{f}}\\
A' \ar@<-.5ex>[u]_-{\overline{a}} \ar@<-.5ex>[r]_-{f'} & B' \ar@<-.5ex>[u]_-{\overline{b}} \ar@<-.5ex>[l]_-{\overline{f'}}}}
\end{equation}
Write $A_{-1}=B'$ and $A_{0}=A$,
\[
\del_{0}=f'\comp a=b\comp f\colon{A_{0}\to A_{-1}}
\]
and $\sigma_{-1}=\overline{a}\comp\overline{f'}=\overline{f}\comp \overline{b}\colon{A_{-1}\to A_{0}}$; then already $\del_{0}\comp \sigma_{-1}=1_{A_{-1}}$. Now consider the extension $a\times_{1_{B'}}f$, which is defined by pulling back the double extension $( f'\comp a,f')\colon{a\to 1_{B'}}$ along the double extension $(f'\comp a, b)\colon{f\to 1_{B'}}$. Hence we can form the following pullback, which defines the morphisms $\del_{0}$ and $\del_{1}\colon{A_{1}\to A_{0}}$.
\[
\xymatrix@=3em{A_{1} \ar@{}[rd]|<{\pullback} \ar[r]^-{p} \ar[d]_-{\langle \del_{0},\del_{1}\rangle } & A_{0} \ar[d]^-{\langle a,f\rangle }\\
A_{0}\times_{A_{-1}}A_{0} \ar[r]_-{a\times_{1_{B'}} f} & A'\times_{B'}B}
\]
We see that
\[
\del_{0}\comp\del_{0}=f'\comp a\comp\del_{0}=f'\comp a\comp p=b\comp f\comp p=b\comp f\comp\del_{1}=\del_{0}\comp\del_{1}.
\] 
Write $\sigma_{0}\colon{A_{0}\to A_{1}}$ for the arrow universally induced by the equality
\[
( a\times_{1_{B'}} f)\comp\langle 1_{A_{0}},1_{A_{0}}\rangle =\langle a,f\rangle \comp 1_{A_{0}};
\]
then $\del_{0}\comp \sigma_{0}=\del_{1}\comp \sigma_{0}=1_{A_{0}}$. Finally, let $\sigma_{-1}\colon{A_{0}\to A_{1}}$ be the arrow universally induced by the equality
\[
( a\times_{1_{B'}} f)\comp\langle 1_{A_{0}},\overline{a}\comp\overline{f'}\comp f'\comp a\rangle =\langle a,f\rangle \comp (\overline{a}\comp a).
\]
Then $\del_{0}\comp \sigma_{-1}=1_{A_{0}}$ and $\del_{1}\comp\sigma_{-1}=\overline{a}\comp\overline{f'}\comp f'\comp a=\sigma_{-1}\comp\del_{0}$. As both $\del_0$ and $\del_1$ are split epimorphisms, they are in $\E$.

The diagram~\eqref{Trunc} thus defined can be extended to a contractible augmented $\E$-simplicial object $\AA$ by constructing successive simplicial kernels, which exist by assumption. This contractible augmented $\E$-simplicial object is $\E$-Kan, so by Proposition~\ref{Proposition-Contractible-plus-Kan-is-Resolution} it is an $\E$-resolution. In particular, the induced comparison morphism $\langle\del_{0},\del_{1}\rangle\colon{A_{1}\to K_{1}\AA}$ is in $\E$. Using (E4) on the square defining~$\langle\del_0,\del_1\rangle$, we see that $\langle a, f\rangle$ is also in $\E$, which means that the split epimorphism of split epimorphisms~\eqref{Split-Epi-of-Split-Epis} is a double extension. This proves that~$(\A,\E)$ satisfies~(E5).
\end{proof}

\begin{theorem}\label{Theorem-Relative-Maltsev}
Let $\A$ be a category with simplicial kernels and $\E$ a class of morphisms in $\A$ satisfying (E1)--(E4). Then the following are equivalent:
\begin{enumerate}
\item[$\cdot$] $(\A,\E)$ satisfies (E5);
\item[$\cdot$] every quasi-simplicial object in $\A$ is $\E$-Kan;
\item[$\cdot$] every simplicial object in $\A$ is $\E$-Kan.\noproof
\end{enumerate}
\end{theorem}

\subsection*{Some examples}
We start with a prototypical example: regular Mal'tsev categories.

\begin{example}[Regular Mal'tsev categories]\label{Example-Relative-Maltsev-Absolute}
It is shown in \cite{Bourn2003} that when $\A$ is finitely complete with coequalisers of effective equivalence relations and $\E$ is the class of regular epimorphisms, the pair $(\A,\E)$ satisfies (E1)--(E5) if and only if~$\A$ is regular Mal'tsev. Alternatively, this result follows from the above together with~\cite[Theorem~4.2]{Carboni-Kelly-Pedicchio}.

More generally, when $\A$ is finitely complete, it was shown in~\cite{Bourn1996} that $\A$ is Mal'tsev (i.e., every reflexive relation in $\A$ is an equivalence relation) if and only if Condition~(iii) of Proposition~\ref{Proposition-(E5)} holds for $\E$ the class of strong (=~extremal) epimorphisms. Given a pair $(\A,\E)$ which satisfies (E1)--(E5), this implies that~$\A$ is Mal'tsev as soon as $\E$ is contained in the class of strong epimorphisms.
\end{example}

\begin{example}[Higher extensions]
Proposition~\ref{(E1)--(E5)-go-up} implies that $(\Ext\A,\E^1)$, the category of extensions (= regular epimorphisms) in a regular Mal'tsev category~$\A$ together with the double extensions, also satisfies the axioms (E1)--(E5), as do all other $(\Extn\A,\E^n)$.
\end{example}

\begin{example}[Naturally Mal'tsev categories]\label{Example-Naturally-Maltsev}
By a result in~\cite{Bourn1996}, a category is \defn{naturally Mal'tsev}~\cite{Johnstone:Maltsev} when, given a split epimorphism of split epimorphisms as in Diagram~\eqref{Double-Split-Epi}, if it is a (downward) pullback of split epimorphisms, then it is an (upward) pushout of split monomorphisms. If now $\A$ is a naturally Mal'tsev category and $\E$ is its class of split epimorphisms, then it is easily seen that Condition~(iii) in Proposition~\ref{Proposition-(E5)} holds. It is then obvious that $(\A,\E)$ satisfies (E1)--(E5). However, the opposite implication does not hold. For instance, the dual of the category of pointed sets is a regular Mal'tsev category in which every regular epimorphism is split but which is not naturally Mal'tsev. 
\end{example}

Now we give two examples where $\E$ need not be contained in the class of regular epimorphisms of $\A$.

\begin{example}[Weakly Mal'tsev categories]\label{Example-Weakly-Maltsev}
A category is said to be \defn{weakly Mal'tsev}~\cite{NMF1} when it has pullbacks of split epimorphisms and the following property holds: in any split epimorphism of split epimorphisms such as Diagram~\eqref{Double-Split-Epi} which is a (downwards) pullback, the splittings $\overline{a}$ and $\overline{f_{1}}$ are jointly epic.  

Let $\A$ be a category and $\E$ a class of epimorphisms in $\A$ such that the axioms (E1)--(E5) hold. Then $\A$ is weakly Mal'tsev as soon as $\A$ has either pushouts of split monomorphisms or equalisers. Indeed, in the first case, consider the diagram
\[
\vcenter{\xymatrix@=3em{P \ar@{.>}[rd] \ar@/_2ex/@<-.5ex>[rdd] \ar@/^2ex/@<.5ex>[rrd] \\
&A_{1} \ar@{}[rd]|<<{\pullback} \ar@<-.5ex>[r]_-{f_{1}} \ar@<-.5ex>[d]_-{a} & B_1 \ar@/_2ex/@<.5ex>[llu]^-{\widetilde{f_{0}}} \ar@<-.5ex>[d]_-{b} \ar@<-.5ex>[l]_-{\overline{f_{1}}}\\
&A_{0} \ar@/^2ex/@<-.5ex>[luu]_-{\widetilde{b}} \ar@<-.5ex>[u]_-{\overline{a}} \ar@<-.5ex>[r]_-{f_{0}} & B_{0} \ar@<-.5ex>[u]_-{\overline{b}} \ar@<-.5ex>[l]_-{\overline{f_{0}}}}}
\]
in which the square is a pullback of $f_{0}$ and $b$ and $P$ is a pushout of $\overline{f_{0}}$ and~$\overline{b}$. Then~$\widetilde{b}$ and $\widetilde{f_{0}}$ are jointly (strongly) epic, and by Proposition~\ref{Proposition-(E5)}  the dotted comparison morphism is also an epimorphism. It follows that the splittings $\overline{a}$ and~$\overline{f_{1}}$ in the pullback are jointly epic. 

In the second case, given two parallel morphisms which coequalise $\overline{a}$ and~$\overline{f_{1}}$, their equaliser ${P\to A_{1}}$ induces a diagram such as above. Then this morphism is both epic and regular monic, so that the two given parallel morphisms are equal to each other.

Conversely, for any pair $(\A,\E)$, where $\A$ is a weakly Mal'tsev category and~$\E$ is the class of all epimorphims, the conditions (E1), (E3) and (E4) hold, but for~(E2) we need epimorphisms in~$\A$ to be pullback-stable. In this case Proposition~\ref{Proposition-(E5)} tells us that $(\A,\E)$ satisfies (E5). A concrete situation where this occurs is given in Example~\ref{Example-Inverse-Image}.
\end{example}

\begin{example}[All morphisms as extensions]
For any category with pull\-backs, a trivial example is obtained by taking $\E$ to be the class of all morphisms.
\end{example}

The following two examples satisfy a stronger axiom, cf.~\cite{Bourn2001,Tomasthesis,EverHopf,Tamar_Janelidze}. 
\begin{enumerate}
\item[(E5\plus)] Given a diagram in $\A$
\[
\xymatrix{0 \ar[r] & \K[a] \ar[r] \ar[d]_-k & A_1 \ar[r]^-{a} \ar[d]_{f} & A_{0} \ar@{=}[d] \ar[r] & 0\\
0 \ar[r] & \K[b] \ar[r] & B \ar[r]_-{b} & A_{0} \ar[r] & 0}
\]
with short exact rows and $a$ and $b$ in $\E$, if $k\in\E$ then also $f\in\E$.
\end{enumerate}
Notice that Axiom (E2) ensures the existence of kernels of extensions. Axiom (E5\plus) implies (E5): consider a split epimorphism of extensions as in (E5). Take kernels of $a$ and $b$ to obtain a split epimorphism of short $\E$-exact sequences:
\[
\xymatrix@C=4em{0 \ar[r] & \K[a] \ar[r]^{\ker a} \ar@<-.5ex>[d]_-k & A_1 \ar[r]^{a} \ar@<-.5ex>[d]_{f_{1}} & A_{0} \ar@<-.5ex>[d]_{f_{0}} \ar[r] & 0\\
0 \ar[r] & \K[b] \ar@<-.5ex>[u] \ar[r]_{\ker b} & B_{1} \ar@<-.5ex>[u] \ar[r]_{b} & B_0 \ar@<-.5ex>[u] \ar[r] & 0}
\]
As $k$ is a split epimorphism and thus in $\E$, (E5\plus) implies that the right hand square is a double extension.

\begin{example}[Topological groups 1]\label{Topological-Groups}
Example~\ref{Example-Topological-Groups} of topological groups and morphisms split in the category of topological spaces satisfies (E1)--(E4), as commented earlier. This example also satisfies the axiom (E5\plus) and hence~(E5). Consider a diagram in $\Gp\Top$ as in (E5\plus), and assume that in $\Top$ the morphism~$k$ is split by a continuous map $u\colon{\K[b]\to \K[a]}$, $a$ is split by $s$ and~$b$ is split by $t=f \comp s$. Any element~$\beta$ in the domain of an extension $b\colon{B\to A_{0}}$ can be written as a product of an element~$\kappa$ of the kernel~$K[b]$ with an element $tb(\beta)$ in the image of the splitting~$t$, because $\beta=\beta\cdot (tb(\beta))^{-1}\cdot tb(\beta)$. We show that the morphism $f\colon{A_1\to B}$ is also split in~$\Top$. A splitting ${B\to A_1}$ is given by the composite
\[
\xymatrix@!0@C=13em@R=2em{B \ar[r] & \K[b]\times A_{0} \ar[r] & A_1\\
\beta \ar@{|->}[r] & (\beta\cdot (tb(\beta))^{-1},b(\beta)) \ar@{|->}[r] & u(\beta\cdot (tb(\beta))^{-1})\cdot sb(\beta)}
\]
which is easily seen to be continuous.
\end{example}

\begin{example}[Rings and modules]
The category of rings together with morphisms split in abelian groups and the category of $R$-modules with morphisms split in $\Ab$ also satisfy the axioms (E1)--(E4) and (E5\plus), and thus~(E5). 
\end{example}

Let~$\A$ be a category with pullbacks, $(\B,\F)$ a pair which satisfies (E1)--(E5) and $U\colon{\A\to \B}$ a pullback-preserving functor. Then the class of morphisms in~$\A$ given by $\E=U^{-1}\F$ gives a pair $(\A,\E)$ which also satisfies (E1)--(E5). The following examples are instances of this situation.

\begin{example}[Topological groups 2]
Using the above, the category of topological groups may be equipped with another class of extensions, different from the one considered in Example~\ref{Topological-Groups}, but such that (E1)--(E5) still hold: let $U$ be the forgetful functor ${\Gp\Top\to \Gp}$ and take $\E=U^{-1}\F$ with $\F$ the class of all regular epimorphisms in $\Gp$.
\end{example}

\begin{example}[Reflective subcategories]
Another instance of this occurs when the functor~$U$ is the inclusion of a reflective subcategory; hence any class of extensions satisfying (E1)--(E5) restricts to any reflective subcategory where it still satisfies (E1)--(E5).
\end{example}

\begin{example}[Weakly Mal'tsev but not Mal'tsev]\label{Example-Inverse-Image}
Finally let $\A$ be the category of sets equipped simultaneously with a group structure and a topology, and morphisms which are continuous group homomorphisms. (We are not assuming any compatibility between the group structure and the topology as in the case of $\Gp\Top$.) Consider the forgetful functor to $\Gp$; then the class of extensions $\E$ induced by the regular epimorphisms of groups, i.e., the continuous surjective homomorphisms in $\A$, satisfies the conditions (E1)--(E5). On the other hand, $\A$ is not a Mal'tsev category in the absolute sense (though it is weakly Mal'tsev). The regular epimorphisms in $\A$ are in particular quotients (inducing the final topology on the codomain) so that not every extension is a regular epi. As a counterexample to the absolute Mal'tsev property, consider the group of integers $\ZZ$ with the indiscrete topology. Then~$\ZZ\times\ZZ$ also carries the indiscrete topology, while $\ZZ+\ZZ$ carries the final topology for the (algebraic) coproduct inclusions. Now the universally induced comparison morphism ${\ZZ+\ZZ\to \ZZ\times \ZZ}$ to the pullback in the split epimorphism of regular epimorphisms
\[
 \xymatrix{\ZZ+\ZZ \ar[r] \ar[d] & \ZZ \ar[d]\\
 \ZZ \ar[r] & 0}
\]
is not a regular epimorphism, as the topology on $\ZZ\times \ZZ$ is different from the induced quotient topology. To see this, it suffices to note that the single\-ton~$\{(1,1)\}$ is not open in $\ZZ\times \ZZ$, whereas its inverse image along ${\ZZ+\ZZ\to \ZZ\times \ZZ}$ is open in~$\ZZ+ \ZZ$.
\end{example}

\subsection*{Acknowledgements}
We would like to thank Tamar Janelidze for fruitful discussion on the subject of this paper, Zurab Janelidze for the counterexample in Example~\ref{Example-Naturally-Maltsev}, and the DPMMS at the University of Cambridge for its kind hospitality during our stay.

%\bibliography{tim}

\begin{thebibliography}{10}

\bibitem{Barr-Beck}
M.~Barr and J.~Beck, \emph{Homology and standard constructions}, Seminar on
  triples and categorical homology theory ({ETH}, {Z\"u}rich, 1966/67), Lecture
  Notes in Math., vol.~80, Springer, 1969, pp.~245--335.

\bibitem{Barr-Grillet-vanOsdol}
M.~Barr, P.~A. Grillet, and D.~H. van Osdol, \emph{Exact categories and
  categories of sheaves}, Lecture Notes in Math., vol. 236, Springer, 1971.

\bibitem{Bourn1996}
D.~Bourn, \emph{Mal'cev categories and fibration of pointed objects}, Appl.
  Categ. Structures \textbf{4} (1996), 307--327.

\bibitem{Bourn2001}
D.~Bourn, \emph{{$3\times 3$} {L}emma and protomodularity}, J.~Algebra
  \textbf{236} (2001), 778--795.

\bibitem{Bourn2003}
D.~Bourn, \emph{The denormalized {$3\times 3$} lemma}, J.~Pure Appl. Algebra
  \textbf{177} (2003), 113--129.

\bibitem{Brown-Ellis}
R.~Brown and G.~J. Ellis, \emph{Hopf formulae for the higher homology of a
  group}, Bull. Lond. Math. Soc. \textbf{20} (1988), no.~2, 124--128.

\bibitem{Carboni-Kelly-Pedicchio}
A.~Carboni, G.~M. Kelly, and M.~C. Pedicchio, \emph{Some remarks on {M}altsev
  and {G}oursat categories}, Appl. Categ. Structures \textbf{1} (1993),
  385--421.

\bibitem{Donadze-Inassaridze-Porter}
G.~Donadze, N.~Inassaridze, and T.~Porter, \emph{{$n$-F}old {{\v C}}ech derived
  functors and generalised {H}opf type formulas}, K-Theory \textbf{35} (2005),
  no.~3--4, 341--373.

\bibitem{Tomasthesis}
T.~Everaert, \emph{An approach to non-abelian homology based on {C}ategorical
  {G}alois {T}heory}, Ph.D. thesis, Vrije Universiteit Brussel, 2007.

\bibitem{EveraertCommutator}
T.~Everaert, \emph{Relative commutator theory in varieties of {$\Omega$}-groups},
  J.~Pure Appl. Algebra \textbf{210} (2007), 1--10.

\bibitem{EverHopf}
T.~Everaert, \emph{Higher central extensions and {H}opf formulae}, J.~Algebra
  \textbf{324} (2010), 1771--1789.

\bibitem{EGAM}
T.~Everaert and M.~Gran, \emph{Relative commutator associated with varieties of
  $n$-nilpotent and of $n$-solvable groups}, Arch. Math. (Brno) \textbf{42}
  (2007), 387--396.

\bibitem{Everaert-Gran-nGroupoids}
T.~Everaert and M.~Gran, \emph{Homology of $n$-fold groupoids}, Theory Appl. Categ. \textbf{23}
  (2010), no.~2, 22--41.

\bibitem{Everaert-Gran-TT}
T.~Everaert and M.~Gran, \emph{Protoadditive functors, derived torsion theories and homology},
  Preprint, 2011.

\bibitem{EGVdL}
T.~Everaert, M.~Gran, and T.~Van~der Linden, \emph{Higher {H}opf formulae for
  homology via {G}alois {T}heory}, Adv.~Math. \textbf{217} (2008), no.~5,
  2231--2267.

\bibitem{EverVdL3}
T.~Everaert and T.~Van~der Linden, \emph{A note on double central extensions in
  exact {Mal'tsev} categories}, Cah. Topol. G{\'e}om. Differ. Cat{\'e}g.
  \textbf{LI} (2010), 143--153.

\bibitem{EverVdL4}
T.~Everaert and T.~Van~der Linden, \emph{Galois theory and commutators}, Algebra Universalis \textbf{65}
  (2011), no.~2, 161--177.

\bibitem{EverVdLRCT}
T.~Everaert and T.~Van~der Linden, \emph{Relative commutator theory in semi-abelian categories}, J.~Pure
  Appl. Algebra \textbf{216} (2012), no.~8--9, 1791--1806.

\bibitem{Juliathesis}
J.~Goedecke, \emph{Three viewpoints on semi-abelian homology}, Ph.D. thesis,
  University of Cambridge, 2009,
  \texttt{www.dspace.cam.ac.uk/handle/1810/224397}.

\bibitem{GVdL2}
J.~Goedecke and T.~Van~der Linden, \emph{On satellites in semi-abelian
  categories: Homology without projectives}, Math. Proc. Cambridge Philos. Soc.
  \textbf{147} (2009), no.~3, 629--657.

\bibitem{Gran-VdL}
M.~Gran and T.~Van~der Linden, \emph{On the second cohomology group in
  semi-abelian categories}, J.~Pure Appl.\ Algebra \textbf{212} (2008),
  636--651.

\bibitem{Hochschild-Wigner}
G.~Hochschild and D.~Wigner, \emph{Abstractly split group extensions}, Pacific
  J.~Math. \textbf{68} (1977), no.~2, 447--453.

\bibitem{Janelidze:Double}
G.~Janelidze, \emph{What is a double central extension? ({T}he question was
  asked by {R}onald {B}rown)}, Cah. Topol. G{\'e}om. Differ. Cat{\'e}g.
  \textbf{XXXII} (1991), no.~3, 191--201.

\bibitem{Janelidze:Hopf-talk}
G.~Janelidze, \emph{Higher dimensional central extensions: A categorical approach to
  homology theory of groups}, Lecture at the International Category Theory
  Meeting CT95, Halifax, 1995.

\bibitem{Janelidze:Hopf}
G.~Janelidze, \emph{Galois groups, abstract commutators and {H}opf formula}, Appl.
  Categ. Structures \textbf{16} (2008), 653--668.

\bibitem{Janelidze-Kelly}
G.~Janelidze and G.~M. Kelly, \emph{Galois theory and a general notion of
  central extension}, J.~Pure Appl. Algebra \textbf{97} (1994), no.~2,
  135--161.

\bibitem{Janelidze-Sobral-Tholen}
G.~Janelidze, M.~Sobral, and W.~Tholen, \emph{Beyond {B}arr exactness:
  Effective descent morphisms}, {C}ategorical Foundations: Special Topics in
  Order, Topology, Algebra and Sheaf Theory (M.~C. Pedicchio and W.~Tholen,
  eds.), Encyclopedia of Math. Appl., vol.~97, Cambridge Univ. Press, 2004,
  pp.~359--405.

\bibitem{Tamar_Janelidze}
T.~Janelidze, \emph{Relative homological categories}, J.\ Homotopy Relat.\
  Struct. \textbf{1} (2006), no.~1, 185--194.

\bibitem{Tamar-Janelidze-Thesis}
T.~Janelidze, \emph{Foundation of relative non-abelian homological algebra}, Ph.D.
  thesis, University of Cape Town, 2009.

\bibitem{Tamar_Janelidze_Semiabelian}
T.~Janelidze, \emph{Relative semi-abelian categories}, Appl. Categ. Structures
  \textbf{17} (2009), 373--386.

\bibitem{Johnstone:Maltsev}
P.~T. Johnstone, \emph{Affine categories and naturally {M}al'cev categories},
  J.~Pure Appl. Algebra \textbf{61} (1989), 251--256.

\bibitem{Johnstone:Elephant}
P.~T. Johnstone, \emph{Sketches of an elephant: A topos theory compendium}, Oxford
  Logic Guides, vol. 43, 44, Oxford Univ. Press, 2002, Volumes 1 and 2.

\bibitem{MacLane}
S.~{Mac\,Lane}, \emph{Categories for the working mathematician}, second ed.,
  Grad. Texts in Math., vol.~5, Springer, 1998.

\bibitem{MacLane-Moerdijk}
S.~{Mac\,Lane} and I.~Moerdijk, \emph{Sheaves in geometry and logic: A first
  introduction to topos theory}, Springer, 1992.

\bibitem{NMF1}
N.~Martins-Ferreira, \emph{Weakly {M}al'cev categories}, Theory Appl. Categ.
  \textbf{21} (2008), no.~6, 91--117.

\bibitem{Moore:Kan}
J.~C. Moore, \emph{Homotopie des complexes mono{\"i}daux}, S{\'e}minaire Henri
  Cartan \textbf{7} (1954--1955), no.~2, Exp. No. 18, 8 p.

\bibitem{RVdL}
D.~Rodelo and T.~Van~der Linden, \emph{The third cohomology group classifies
  double central extensions}, Theory Appl. Categ. \textbf{23} (2010), no.~8,
  150--169.

\bibitem{Tierney-Vogel}
M.~Tierney and W.~Vogel, \emph{Simplicial derived functors}, Category Theory,
  Homology Theory and their Applications {I}, Lecture Notes in Math., vol.~86,
  Springer, 1969, pp.~167--180.

\bibitem{Tierney-Vogel2}
M.~Tierney and W.~Vogel, \emph{Simplicial resolutions and derived functors}, Math. Z.
  \textbf{111} (1969), no.~1, 1--14.

\bibitem{Wigner-Topological-Groups}
D.~Wigner, \emph{Algebraic cohomology of topological groups}, Trans. Amer.
  Math. Soc. \textbf{178} (1973), 83--93.

\end{thebibliography}
%\bibliographystyle{amsplain}

\end{document}